\documentclass{article}
 
\usepackage[latin1]{inputenc}
\usepackage[T1]{fontenc}
\usepackage[francais, english]{babel}
\usepackage{lmodern}
\usepackage{stmaryrd}
\usepackage{amsthm}
\usepackage{amsmath}
\usepackage{amssymb}
\usepackage{mathrsfs}
\usepackage{soul}
\usepackage{layout}
\usepackage{setspace}
\usepackage[top=2.7cm, bottom=2.7cm, left=3.3 cm, right=3.3 cm]{geometry}
\usepackage{pgf,tikz}
\usepackage{graphicx}
\usepackage{xcolor}
\usetikzlibrary{arrows}
\usepackage{listings}
\usepackage{enumitem}
\usepackage{hyperref}
\usepackage{cases}
\usepackage{mathabx}
\usepackage{cleveref}
\usepackage{xcolor}
\usepackage{appendix}
\usepackage{fancyhdr}
\usepackage{sectsty}
\usepackage{comment}

\usepackage{caption}
\usepackage{graphicx, threeparttable}
 \DeclareCaptionFormat{sanslabel}{#3}%

\usepackage{tikz-cd}

\newtheorem{thm}{Theorem}[section]
\newtheorem{prop}[thm]{Proposition}
\newtheorem{cor}[thm]{Corollary}
\newtheorem{lemma}[thm]{Lemma}
\newtheorem{fact}[thm]{Fact}

\newtheorem*{th*}{Theorem}
\newtheorem*{prop*}{Proposition}
\newtheorem*{cor*}{Corollary}
\newtheorem*{lemma*}{Lemma} 
\newtheorem*{def*}{Definition}
\newtheorem*{rem*}{Remark}

\newtheorem*{fact*}{Fact}

\newcommand{\Hn}{\mathbb{H}^N}
\newcommand{\R}{\mathbb{R}}
\newcommand{\Z}{\mathbb{Z}}

\newcommand{\N}{\mathbb{N}}

\newcommand{\HH}{\mathcal{H}}

\newcommand{\GG}{\Omega}

\newcommand{\eps}{\varepsilon}

\newcommand{\E}{ \mathbb{E} }

\newcommand{\supp}{\text{supp }}

\newcommand{\tsigma}{\widetilde{\sigma}}
\newcommand{\GGb}{\overline{\GG}^B}
\newcommand{\BG}{X}
\newcommand{\x}{\overline{x}}
\newcommand{\y}{\overline{y}}

\newcommand{\LL}{\text{LL}}
\newcommand{\Pb}{\mathbb{P}}
\newcommand{\wconv}{\overset{dist}{\longrightarrow}}
\newcommand{\pconv}{\overset{prob}{\longrightarrow}}

\newcommand{\details}[1]{{\color{black}#1}}

\subsectionfont{\fontsize{13}{17}\selectfont}

\subsubsectionfont{\fontsize{12}{15}\selectfont}

\makeatletter
\newcounter {subsubsubsection}[subsubsection]
\renewcommand\thesubsubsubsection{\thesubsubsection .\@alph\c@subsubsubsection}
\newcommand\subsubsubsection{\@startsection{subsubsubsection}{4}{\z@}%
                                     {-3.25ex\@plus -1ex \@minus -.2ex}%
                                     {1.5ex \@plus .2ex}%
                                     {\normalfont\large\bfseries}}
\newcommand*\l@subsubsubsection{\@dottedtocline{3}{10.0em}{4.1em}}
\newcommand*{\subsubsubsectionmark}[1]{}
\makeatother

\title{Winding of geodesic rays chosen by a harmonic measure}
\author{Timothée Bénard\thanks{
The author has received funding from the European Research
Council (ERC) under the European Union's Horizon 2020 research and
innovation programme (grant agreement No. 803711).}}
\date{}

\begin{document}

\maketitle

{
\abstract{We prove limit theorems for the homological winding of geodesic rays distributed via a harmonic  measure on a Gromov hyperbolic space. We obtain applications  to the inverse problem for the harmonic measure, and  winding statistics for Patterson-Sullivan measures. }}

\bigskip
\large

\tableofcontents
\bigskip


\section{Introduction} \label{Sec-intro}

\subsection{The inverse problem for the harmonic measure}

Let $\Omega$ be a proper geodesic hyperbolic metric space endowed with a basepoint $o$, and $\Gamma$ be a finitely generated subgroup of isometries of $\Omega$. Given a non-elementary probability measure $\mu$ on $\Gamma$, it is known that for $\mu^{\otimes \N^*}$-almost every $b=(b_{i})_{i\geq 1}\in \Gamma^{\N^*}$, the sample path $(b_{1}\dots b_{n}.o)_{n\geq 0}$ on $\Omega$ converges to a point $\xi_{b}$ in the Gromov boundary $\partial \Omega$. The distribution $\nu$ of the exit point $\xi_{b}$ is called the \emph{harmonic measure} or the \emph{Furstenberg measure} of generator $\mu$. It is the only probability measure on $\partial \Omega$ satisfying the stationarity equation
$$\mu*\nu=\nu$$
The properties of regularity of $\nu$ are very much studied \cite{Guivarch90, Li18} \cite{HochmanSolomyak17, Ledrappier83, LedrappierLessa21} \cite{Bourgain12, BQ18}, and can be used to  describe the asymptotic properties of $\mu$-sample paths on $\Omega$ \cite{BQRW, GuivarchLepage16, BQ16-CLThyp, Li18}.

The motivation for this text is to address the following inverse problem for the harmonic measure:

\bigskip
\noindent \emph{How much information on $\mu$ can be read in (the measure class of) its harmonic measure $\nu$ on the boundary?} 

\bigskip
We will  shed some light on this question by examining the homological winding of $\nu$-typical geodesic rays on $\Gamma\backslash \Omega$. As the statement of the main result requires a rather substantial introduction, we start with a few remarkable corollaries. The main result, as well as standard terminology concerning $\mu$, are  presented  in \Cref{Sec-main-result}. 

\bigskip

\noindent{\bf (I) Expectation and covariance of $\mu_{ab}$}. In the context of random walks on hyperbolic groups, we see that $\nu$ informs on the half lines spanned by the expectation and the covariance of $\mu$ in the abelianization.

     \begin{cor} \label{cor-grhyp}
Let $(\Gamma, d_{\Gamma})$ be a word hyperbolic group equipped with a word metric, $\mu$ a non-elementary measure on $\Gamma$, and $\nu$ its harmonic measure on the Gromov boudary $\partial \Gamma$.  Write $H_{\Gamma}$ the torsion-free part of the abelianization of $\Gamma$, and $\mu_{ab}$ the projection of $\mu$ to $H_{\Gamma}$.  
\begin{itemize}
\item If $\mu$ has finite first moment, then the measure class of $\nu$ determines 
$$\lambda_{\mu}^{-1}\E(\mu_{ab})  $$
where $\lambda_{\mu}>0$ is the rate of escape of the $\mu$-walk on $(\Gamma, d_{\Gamma})$. 

\item If $\mu$ has finite second moment and centered projection to $H_{\Gamma}$, then the measure class of $\nu$ determines 
$$\lambda_{\mu}^{-1}Cov(\mu_{ab})  $$
\end{itemize}
\end{cor}

In the above statement, $\E(\mu_{ab}), Cov(\mu_{ab})$ are respectively the expectation and the covariance of $\mu_{ab}$ in the real vector space spanned by the free abelian group $H_{\Gamma}$. These quantities would not make sense on the torsion part of $\Gamma/[\Gamma, \Gamma]$. 

It is remarkable that these formulas  link objects which have very different origins. The rate of escape $\lambda_{\mu}$ arises from  hyperbolicity, and depends strongly on the  word distance $d_{\Gamma}$ on $\Gamma$. In contrast,  $\mu_{ab}$ is independent from the metric and lives in a commutative setting.  

\Cref{cor-grhyp} is a direct consequence of \Cref{laws-projected} below and the formulas for the drift and the limit covariance.

 \bigskip
 \noindent{\bf (II) Non-centered walks are singular at infinity}.  
 Assume $\Omega$ is the real hyperbolic space $\Hn$ of dimension $N\geq 2$, and $\Gamma$ is  a  convex cocompact non-elementary discrete subgroup of orientation preserving isometries. An established question is to understand the properties of measures $\mu$ for which $\nu$ is in the class of Patterson-Sullivan on $\partial \Hn$. Note that such measures do exist by a result of  Connell-Muchnik \cite{ConnellMuchnik07}, and can even be chosen with finite exponential moment by results of Lalley and Li \cite{Lalley86, LiNaudPan21}. An active conjecture, commonly attributed to Kaimanovich-Le Prince, states that $\mu$ cannot be finitely supported \cite{KaimanovichLePrince11, Kosenko21,  KosenkoTiozzo20}. We prove here that $\mu$ must be centered. 
 
\begin{cor}\label{non-centered-sing}
  Let $\Gamma\leq Isom^+(\Hn)$ be a  convex cocompact discrete subgroup,  $\mu$  a non-elementary probability measure on $\Gamma$ with finite first moment in a word metric. If the harmonic measure of $\mu$ is in the class of Patterson-Sullivan, then the projection of $\mu$ to the torsion-free part of $\Gamma/[\Gamma, \Gamma]$ is centered. 
     \end{cor}

\Cref{non-centered-sing} is a direct consequence of \Cref{laws-projected} and \Cref{zero-drift} below.

\bigskip
 \noindent{\bf (III)  Walks with cusps are singular at infinity}.
Furstenberg's discretization of the Brownian motion  \cite{Furstenberg71, Furstenberg73} implies that any (possibly non-uniform) lattice of $PSL_{2}(\R)$ is the support of a probability measure $\mu$ whose harmonic measure is Lebesgue on $\partial \mathbb{H}^2\simeq S^1$. The proof gives $\mu$ with finite first moment in the metric induced by $PSL_{2}(\R)$.  It came thus as a surprise when Guivarc'h-Le Jan  \cite{GuiLeJ90, GuiLeJ93} proved twenty years later that $\mu$ cannot have finite first moment in the \emph{word metric} of $\Gamma$ when $\Gamma$ is the congruence subgroup $\Gamma(2)\leq PSL_{2}(\Z)$. The result was extended step by step to any lattice  $\Gamma\leq PSL_{2}(\R)$ containing a parabolic element \cite{DKN09, GadreMaherTiozzo15}. Further generalisations were made in  \cite{RandeckerTiozzo21, GekhtmanTiozzo20}, notably the case where $\Gamma\leq PSL_{2}(\R)$ is  any discrete finitely generated subgroup with a parabolic element, which was solved under the additional assumptions that $\mu$ generates $\Gamma$ as a semigroup and has superexponential moment \cite{GekhtmanTiozzo20}. We show those two  assumptions are unnecessary. 
 
  \begin{cor} \label{singularity}
 Let $\Gamma\subseteq PSL_{2}(\R)$ be a finitely generated discrete subgroup containing a parabolic element. Let $\mu$ be  a non-elementary  probability measure on $\Gamma$ with finite first moment in a word metric. Then the Furstenberg measure of $\mu$ on $\partial \mathbb{H}^2$ is singular with the class of Patterson-Sullivan. 
\end{cor}
 
The proof of \Cref{singularity} is given in \Cref{Sec-sing}. It relies on \Cref{laws-projected-weak} and a result of Enriquez-Franchi-Le Jan \cite{EnriquezFranchiLeJan01}.

 \subsection{Homological winding of $\nu$-typical rays}
\label{Sec-main-result}

We address the inverse problem for the harmonic measure by reading information about $\mu$ in  the winding statistics of $\nu$-typical geodesic rays on $\Gamma\backslash \Omega$. More precisely, for every $\xi\in \partial \Omega$, call $r_{\xi}:\R^+\rightarrow \Omega$ a geodesic ray   from $o$ to $\xi$ on $\Omega$. Letting $\xi$ vary with law $\nu$,  we obtain a family of random segments $(r_{\xi}([0,t]))_{t\in \R^+}$ indexed by the time parameter. We prove  limit laws for the winding  of $(r_{\xi}([0,t]))_{t\in \R^+}$ on $\Gamma \backslash \Omega$.  In those limit theorems, natural numerical quantities appear (drift, limit covariance) and have explicit formulas involving $\mu$. Those formulas only depend on $\nu$ so they yield numerical invariants for all possible generators $\mu$ of $\nu$.  This strategy is motivated by the note of Guivarc'h-Le Jan  \cite{GuiLeJ90}. It is also related to Gadre-Maher-Tiozzo \cite{GadreMaherTiozzo15} and Randecker-Tiozzo \cite{RandeckerTiozzo21}, as we discuss below.

 \subsection*{Homological projection}

The winding number of a geodesic segment $[o,z]$ on $\Gamma \backslash \Omega$ will be formally incarnated by a vector $i(z)$ where $i$ is a $\Gamma$-equivariant map from $\Omega$ to a vector space. Examples below justify the relation with homology and the term winding.  

 \begin{def*}
 A  measurable map  $i: \Omega\rightarrow \R^d$ ($d\geq 1$) is   \emph{$\Gamma$-equivariant} if there exists a  morphism $\pi: \Gamma\rightarrow \R^d$ satisfying  for all $\gamma\in \Gamma$, $x\in \Omega$, 
 $$i(\gamma.x)=\pi(\gamma)+i(x) $$
 \end{def*}
 \noindent Note that $\pi$ is unique.  

\bigskip

We  consider two natural  levels of regularity on $i$:

\begin{itemize}
\item
We say $i$ is \emph{locally bounded} if for every  ball $B$ of $\Omega$, we have $\sup_{B} \|i\|<\infty$.  
 \item
 We say $i$ is \emph{quasi-Lipschitz} if there exists $R>0$ such that for all $x,y\in \Omega$, we have  $$\|i(x)-i(y)\|\leq R\,d_{\Omega}(x,y)+R $$
 \end{itemize}

\noindent{\bf Examples}. 
\bigskip

1. (\emph{Abelianization of hyperbolic group}).  Let $\Gamma$ be a word hyperbolic group, and $\pi : \Gamma\rightarrow \R^d$ any group morphism. Given any Cayley graph $\Omega$ of $\Gamma$, we can extend $\pi$ into $i:\Omega\rightarrow \R^d$, say by \details{affine} interpolation. The regularity of $i$ is always Lipschitz.

\bigskip
2. (\emph{Homological winding on Gromov hyperbolic manifolds}).  Let $\Omega$ be a simply connected complete Riemannian manifold which is Gromov hyperbolic,  and $\Gamma$  a subgroup of isometries of $\Omega$. Given a basepoint $o\in \Omega$ and a family ${\omega}_{1}, \dots, {\omega}_{d}$ of $\Gamma$-invariant closed  $1$-forms on $\Omega$, we can  define a 
$\Gamma$-equivariant map $i:\Omega\rightarrow \R^d$ by setting
$$i(z)=\left(\int_{[o,z]}{\omega}_{1}\,,\, \dots \,,\,  \int_{[o,z]}{\omega}_{d}\right)$$
where $[o,z]$ stands for any smooth path from $o$ to $z$. 

Observe that  $i$ is smooth,  whence locally bounded. It can be quasi-Lipschitz, for instance if  $\sup_{x\in \Omega, i\leq d} \|\omega_{i}\|_{x}<\infty$. However, if $\Gamma\backslash \Omega=M$ is a manifold and some $\omega_{i}$ has non zero integral on a loop $l$ of $M$ that may be shrinked arbitrarily by homotopy, then $i$ cannot be quasi-Lipschitz. \details{Indeed, first note that $\omega_{i}$ has the same integral $c\neq 0$ on all the homotopic deformations of $l$ in $M$, then for any $\eps>0$, choose such a deformation of length at most $\eps$, lift the path following this loop $\lfloor \eps^{-1}\rfloor$ times  into a path in $\Omega$ with endpoints $z, z'$  satisfying $d(z, z') \leq 1$, and observe however that $\|i(z)-i(z')\| \geq \lfloor \eps^{-1}\rfloor c \neq 0$ can be made arbitrarily large by choosing $\eps$ small, thus causing $i$ to be non quasi-Lipschitz.} For instance, this occurs if $M$  admits a cusp containing a loop with non trivial homology, as in the hyperbolic $3$-horned  sphere (but not the once-punctured torus). 

When $M=\Gamma\backslash \Omega$ is a manifold, the vector $i(z)$ is related to the \emph{homological winding} of the path $[o,z]$ projected to $M$. Indeed, assume that $\omega_{1}, \dots, \omega_{d}$ project to a basis of $H^1(M, \R)$. This yields a representative $\omega$ for any cohomology class in $H^1(M,\R)$, that we see as a $\Gamma$-invariant closed $1$-form on $\Omega$. Using that the integral pairing $H^1(M, \R)\times H_{1}(M, \R)\rightarrow \R $ is non-degenerate (De Rham Theorem), we then have 
$$i : \Omega \rightarrow H_{1}(M, \R), \,\,z\mapsto \int_{[o,z]} $$
\details{Here we use that we specified representatives for a basis of $H^{1}(M, \R)$ because the integral of an element of $H^{1}(M, \R)$ is not well defined on a arbitrary path in $M$ or $\Omega$.} A priori, $i$ depends on the choice of basepoint $o$ and on the basis representatives $\omega_{i}$. However, any other choice will lead to a $\Gamma$-equivariant map $i'$ sharing the same  morphism $\pi : \Gamma\rightarrow H_{1}(M, \R)$. In this sense, the construction is canonical. The theorems below will not distinguish  $i$ from $i'$.

\bigskip
 \noindent{\bf Historical comments}.  The homological winding for the geodesic flow on a negatively curved manifold has been thoroughly investigated from the point of view of \emph{invariant measures} (say Liouville measure or Bowen-Margulis-Sullivan  measures). For the compact case, see \cite{Sinai60, Ratner73CLT, DenkerPhilipp84}. In presence of cusps, see   \cite{GuiLeJ93, LeJan94}.  Homological winding is also related to the problem of counting closed geodesics as in \cite{KatsudaSunada90}. A more precise form of winding, involving $1$-forms which are only closed in a neighborhood of the cusps, is studied in \cite{EnriquezLeJan97, EnriquezFranchiLeJan01, Franchi99, BabillotPeigne06}. 
 
Homological winding for rays chosen with a \emph{Furstenberg measure} has not been much studied. It is hinted at in the seminal note of Guivarc'h-Le Jan \cite{GuiLeJ90}. It is also related to the works of Gadre-Maher-Tiozzo \cite{GadreMaherTiozzo15} and Randecker-Tiozzo \cite{RandeckerTiozzo21} who consider a typical geodesic ray on a hyperbolic manifold and show that the non-oriented winding around a given cusp satisfies a law of large numbers with finite mean. In this paper, we establish a wide range of limit theorems for the homological winding of Furstenberg-typical geodesic rays. Contrary to \cite{GadreMaherTiozzo15, RandeckerTiozzo21} which rely on Kingman's subbaditive ergodic theorem, we compare directly  geodesic rays with an abelian random walk (or  martingale in the non-centered case) in the homology group. This new approach allows a  precise control of the winding statistics and to free oneself from explicit geometrical computations, yielding with no extra cost the case of an arbitrary hyperbolic geodesic metric space. This  level of generality, notably describing abelianized geodesic rays in the context of a word hyperbolic group, does not seem to have prior history.


 \subsection*{Statistics of $\nu$-typical rays}

We now present the main result of the paper, describing  the winding statistics of geodesic rays on $\Gamma\backslash \Omega$ chosen by a Furstenberg measure $\nu$ of generator $\mu$ supported by $\Gamma$. Several conditions  on $\mu$ will come up:

 \begin{itemize}
  \item  $\mu$ is \emph{non-elementary} if the semigroup generated by its support  contains two hyperbolic isometries with disjoint sets of fixed points.

 \item  $\mu$ has  \emph{finite $p$-th moment}  if  $\sum_{\gamma\in \Gamma} d_{\Gamma}(e, \gamma)^p \mu(\gamma)<\infty$ where $d_{\Gamma}$ is any word distance on $\Gamma$.
 
 \item $\mu$ has \emph{finite exponential moment} if  $\exists s>0$ such that $\sum_{\gamma \in \Gamma}e^{s d_{\Gamma}(e, \gamma)}\mu(\gamma)<\infty$. 
 \end{itemize}

We divide our result in two statements, assuming  different levels of  regularity for the projection.

\bigskip
\Cref{laws-projected} assumes  $i$ to be quasi-Lipschitz. This situation occurs if $i$ is the projection of a hyperbolic group to its abelianization, or the homological projection for some convex cocompact hyperbolic manifold. We prove a law of large numbers (LLN), and a central limit theorem (CLT), a law of the iterated logarithm (LIL), a principle a large deviations (PLD), and a gambler's ruin estimate (GR).

\begin{thm}[Limit laws for  geodesic rays I] \label{laws-projected}  Let $\Omega$ be a proper  geodesic hyperbolic metric space \details{with a basepoint},  and $\Gamma$  a finitely generated subgroup of isometries of $\Omega$.  Fix a \emph{quasi-Lipschitz} $\Gamma$-equivariant map  $i: \Omega\rightarrow \R^d$. Let $\nu$ be a Furstenberg measure on $\partial \Omega$ for some  non-elementary probability measure  $\mu$ on $\Gamma$.
 
\begin{itemize}
\item \emph{(LLN)} If $\mu$ has  finite first moment, then there is a vector $e_{\nu}\in \R^d$ such that for $\nu$-almost every $\xi \in \partial \Gamma$, 
$$\frac{  i \circ r_{\xi}(t)}{t} \underset{t\to+\infty}{\longrightarrow} e_{\nu}$$

\item \emph{(CLT)}  If $\mu$ has  finite second  moment, then as $\xi$ varies with law $\nu$, we have the convergence in distribution, 
\details{$$\frac{i \circ r_{\xi}(t) - t e_{\nu}}{\sqrt{t}} \underset{t\to+\infty}{\wconv}  \mathscr{N}(0, A_{\nu}) $$
}for some  positive semidefinite symmetric  matrix $ A_{\nu} \in M_{d}(\R)$.

\item \emph{(LIL)} If $\mu$ has  finite second  moment, then for $\nu$-almost every $\xi \in \partial \Gamma$, 
$$\left\{\text{ Accumulation points of } \left(\frac{i\circ r_{\xi}(t) - te_{\nu}}{\sqrt{2t \log \log t}}\right)_{t>2}  \right\} = A^{1/2}_{\nu}(\overline{B}(0,1)) $$
where $\overline{B}(0,1)=\{t \in \R^d, \sum_{i}t_{i}^2\leq 1\}$.

\item \emph{(PLD)}  If $\mu$ has  finite exponential  moment, then for all $\alpha>0$, there exists $D, \delta>0$ such that  
$$\nu \{\xi : \|i \circ r_{\xi}(t)- te_{\nu} \|\geq \alpha t \}\leq De^{-\delta t} $$

\item \emph{(GR)}   If $\mu$ has  finite exponential  moment, $d=1$, and $A_{\nu}>0$, then for any $k,l>0$ and any $(k_{s})_{s> 0},(l_{s})_{s> 0} \in \R^{\R_{>0}}$ such that $k_{s}\sim sk$, $l_{s}\sim sl$, we have 
$$\nu \left\{\xi : (i\circ r_{\xi}(t) -t e_{\nu})_{t\geq 0} \text{ leaves  $[-k_{s}, l_{s}]$ via $(l_{s},+\infty)$}   \right\} \,\underset{s\to+\infty}{\longrightarrow}\, \frac{k}{k+l}$$
\end{itemize}
\end{thm}

 \bigskip
 \noindent{\bf Remarks}. 
 \details{1) We recall that the notation $r_{\xi}$ refers to a choice of geodesic ray from the basepoint of $\Omega$ to $\xi\in \partial \Omega$. The theorem does not depend on the choice geodesics rays as any two geodesic rays on $\Omega$ with same starting point and same exit point in the Gromov boundary remain at distance bounded by a constant depending only on $\Omega$.}
 
 2)  In the second statement, we wrote  $\mathscr{N}(0, A_{\nu})$ for the (possibly degenerate) centered gaussian measure on $\R^d$ of covariance matrix $ A_{\nu}$.

 3)  The LLN and the LIL imply that the mean $e_{\nu}$ and the covariance matrix $ A_{\nu}$ only depend on the measure class of $\nu$. We will also give explicit formulas in terms of $\mu$.

 \bigskip

  \Cref{laws-projected-weak}  below is  more general, for  it only assumes  $i$ to be  locally bounded.  As observed earlier in Example 2., it is particularly relevant when $i$ is the  homological projection to a hyperbolic manifold with a cusp. We prove a weak law of large numbers (WLLN), and a  central limit theorem (CLT) encapsulating that of   \Cref{laws-projected}.

\begin{thm}[limit laws for  geodesic rays II] \label{laws-projected-weak} Let $\Omega$ be a proper geodesic hyperbolic metric space \details{with a basepoint},  and $\Gamma$  a finitely generated subgroup of isometries of $\Omega$.  Fix a \emph{locally bounded} $\Gamma$-equivariant map  $i: \Omega\rightarrow \R^d$. Let $\nu$ be a Furstenberg measure on $\partial \Omega$ for some  non-elementary probability measure  $\mu$ on $\Gamma$.\begin{itemize}
\item \emph{(WLLN)} If $\mu$ has  finite first moment,  then as $\xi$ varies with law $\nu$, we have the convergence in probability 
\details{$$\frac{  i \circ r_{\xi}(t)}{t} \underset{t\to+\infty}{\pconv} e_{\nu} $$
}for some  vector $e_{\nu}\in \R^d$.

\item \emph{(CLT)}  If $\mu$ has  finite second  moment, then as $\xi$ varies with law $\nu$, we have the convergence in distribution, 
\details{$$\frac{i \circ r_{\xi}(t) - t e_{\nu}}{\sqrt{t}} \underset{t\to+\infty}{\wconv}  \mathscr{N}(0, A_{\nu}) $$
}for some  positive semidefinite symmetric  matrix $ A_{\nu} \in M_{d}(\R)$.
\end{itemize}
 \end{thm}

\bigskip

 \bigskip

\noindent{\bf  Mean and  covariance matrix}. The mean $e_{\nu}$ and the covariance matrix $A_{\nu}$ appearing in Theorems \ref{laws-projected},  \ref{laws-projected-weak} have  explicit formulas in terms of the measure $\mu$. As $e_{\nu}, A_{\nu}$ only depend on $\nu$, \emph{those formulas give numerical invariants on all possible $\mu$ such that $\mu*\nu=\nu$}. 

To express $e_{\nu}, A_{\nu}$ we need  introduce the \emph{rate of escape} $\lambda_{\mu}$ of the  $\mu$-sample paths on $\Omega$:
$$\lambda_{\mu}=\lim_{n\to +\infty} \frac{d_{\Omega}(o, b_{1}\dots b_{n}o)}{n} $$
 where $d_{\Omega}$ is the distance on $\Omega$, and $(b_{i})_{n\geq 1}\in \Gamma^{\N^*}$ is $\mu^{\otimes \N^*}$-typical. 
 \details{An application of Kingman's subbaditive ergodic theorem to the probability space $(\Gamma^{\N^*}, \mu^{\otimes\N^*})$ endowed with the shift shows that if $\mu$ has finite first moment, then $\lambda_{\mu}$ is well defined and the convergence also holds in $L^1(\Gamma^{\N^*}, \mu^{\otimes\N^*})$. If $\mu$ is also non-elementary, then $\lambda_{\mu}$ is non-zero, see for instance \cite[Proposition 3.3]{BQ16-CLThyp}.}

\bigskip
Letting $\mu_{ab}=\pi_{\star}\mu$, we  prove  
$$ e_{\nu}  =\lambda_{\mu}^{-1}\E(\mu_{ab}) $$

\bigskip

To describe  $A_{\nu}$, we introduce the Busemann boundary   $X$ of $(\Omega,d_{\Omega})$, the Busemann cocycle $\sigma : \Gamma\times X \rightarrow \R$. We prove in \details{\Cref{Sec-proof-CLT}} that uniformly in $x\in X$, 
\begin{equation} \label{formule-covariance}
\lambda_{\mu}A_{\nu}= \lim_{n\to+\infty}\frac{1}{n}\int_{\Gamma} (   \pi(\gamma) -\sigma(\gamma^{-1},x) e_{\nu}   )\, {^t(\pi(\gamma)-\sigma(\gamma^{-1},x) e_{\nu} ) }\,d\mu^{*n}(\gamma)
\end{equation}
 where $v\mapsto {^tv}$ refers to the transposition map.  
 In particular, we get the complement information that 
 
 \begin{itemize} 
 \item \emph{$A_{\nu}$ coincides with the covariance matrix $\lambda_{\mu}^{-1}Cov(\mu_{ab})$   if $\mu_{ab}$ is centered (but not in general).}
 \item \emph{$A_{\nu}$ is non-degenerate if  $\langle \supp \mu_{ab} \rangle_{\text{Vect}}= \R^d$ and the restriction  of the stable length function $l:\Gamma\rightarrow \R^+, \gamma\mapsto \lim_{n\to+\infty} \frac{1}{n}d_{\Omega}(o,\gamma^n.o)$ to the semigroup generated by the support of $\mu$  does not factor via $\pi$ (\Cref{ND-criterion})
 }
 \end{itemize}

\bigskip

The connection between the statistics of   the winding  $i\circ r_{\xi}$ and that of the $\mu_{ab}$-random walk on $\R^d$ arises from the strategy of the proof, which is to approximate $\nu$-typical geodesic rays on $\Omega$ by $\mu$-sample paths then to project them to $\R^d$. 
It came thus as a surprise, while working on this paper, to find that $A_{\nu}$ may not be proportional to $Cov(\mu_{ab})$. Even worse: \emph{it is possible that $Cov(\mu_{ab})$ has rank $d-1$ while $A_{\nu}$ has rank $d$}. This is the case if $\Gamma$ is the free group $F_{2}=\langle u,v\rangle$,   $\pi : F_{2} \rightarrow \Z^2, u, v \mapsto (1,0), (0,1)$, and $\mu=\frac{1}{3}(\delta_{u}+\delta_{uv}+\delta_{uv^{-1}})$,  (see \Cref{exampleAnu}). More explanations on this peculiarity are given in the strategy of proof below.

\subsection{Winding of Patterson-Sullivan measures}
Applying Theorems \ref{laws-projected}, \ref{laws-projected-weak} to the context of Example 2 above, we obtain \emph{winding statistics for the geodesic flow directed by a Furstenberg measure on a  general  manifold $M$ with Gromov hyperbolic universal cover}. The particular case of Patterson-Sullivan measures for  convex cocompact discrete subgroups yields the following. Recall that $\Hn$  is the $N$-dimensional real hyperbolic space.

\begin{cor}[Homological winding for PS-measures] \label{laws-hyperbolic} 
Let $\Gamma \leq Isom^+(\Hn)$  be a convex cocompact torsion-free non-elementary discrete subgroup, and $m$ be any probability measure on $\partial \Hn$ that is absolutely continuous with respect to the class of Patterson-Sullivan of $\Gamma$. 

Then the homological projection of $m$-typical geodesic rays to $H_{1}(\Gamma\backslash \Hn, \R)$ has sublinear escape, satisfies  a central limit theorem with non-degenerate  covariance matrix, the law of the iterated logarithm, and the gambler's ruin estimate (for any $1$-dimensional projection). It also satisfies the  principle of large deviations if $m$ has $L^{1+\eps}$-Radon Nikodym derivatives with respect to Patterson-Sullivan measures.
\end{cor}

The proof is given in \Cref{Sec-laws-hyperbolic}. It combines  \Cref{laws-projected}, the fact that  PS-measures arise as Furstenberg measures for walks with finite exponential moment \cite{LiNaudPan21}, and -to establish sublinear escape- the fact that PS-measures represent the geodesic projection of the Bowen-Margulis-Sullivan measure,  which is  symmetric and flow-invariant on $\Gamma\backslash \Hn$.

 The CLT (for $m$ general), the (multi-dimensional) LIL, the GR, appear to be new, even in the most simple case of real hyperbolic surfaces, i.e. when $N=2$.

The  LLN,  the PLD, and the CLT when $m$ is a Patterson-Sullivan measure, can be obtained by combining results already present in the literature. The seminal case of the CLT on a compact hyperbolic manifold follows from  Sinai's paper \cite{Sinai60}. For the general case, the argument consists in coding the geodesic flow by a suspension flow above a  shift of finite type \cite{Bowen73, ConstLafThom20}, then  apply techniques of symbolic dynamics \cite{Ratner73CLT, DenkerPhilipp84, Waddington96}. Our approach is different, we are directly able to simulate random geodesic rays projected in homology by a random walk on $H_{1}(\Gamma\backslash \Hn, \R)$ with i.i.d. increments, and apply classical results for additive Markov chains.  

We also signal to the reader only interested in  \Cref{laws-hyperbolic} that the proof of this  result, once isolated, is much shorter than the present paper. Indeed, it does not use the appendix whose purpose is to deal  with the non-centered case in Theorems \ref{laws-projected}, \ref{laws-projected-weak}.

\subsection{Strategy and organization of the paper}

The guiding principle to prove Theorems \ref{laws-projected},  \ref{laws-projected-weak} is simple. Let $(B,\beta)=(\Gamma^{\N^*},\mu^{\otimes \N^*})$ denote the space of instructions. Typical  $b\in B$ give a sample path $b_{1}\dots b_{n}.o$ on $\Omega$ that stays close to  the geodesic ray $r_{\xi_{b}}$, whose distribution coincides with $\nu$ (for $b$ varying with law $\beta$).  Hence, we aim to infer the behavior of $i\circ r_{\xi_{b}}(n \lambda_{\mu})$ by comparison with that of $\pi(b_{1}\dots b_{n})$. Unfortunately, $\pi(b_{1}\dots b_{n})$ only approximates $i\circ r_{\xi_{b}}(t_{b,n})$ where $t_{b,n}=d_{\Omega}(o,b_{1}\dots b_{n}.o)$  varies at scale $\sqrt{n}$ around $n\lambda_{\mu}$. This approximation is not precise enough to prove second-order limit theorems. For this reason, we introduce the stopping times $\tau_{n}(b)$ at which $(b_{1}\dots b_{i}.o)_{i\geq 1}$ leaves the ball of radius $n\lambda_{\mu}$, and we study how $\pi(b_{1}\dots b_{\tau_{n}(b)})$ approximates $i\circ r_{\xi_{b}}(n\lambda_{\mu})$. Then we reduce the limit theorems for  $i\circ r_{\xi_{b}}(n\lambda_{\mu})$ to  limit theorems for the  random variables $\pi(b_{1}\dots b_{\tau_{n}(b)})$, and we  obtain the latter  through a careful analysis of the stopping times $\tau_{n}$. The  aforementioned reduction  also requires some work, given that both sequences of variables ultimately do not yield the same covariance matrix. This discrepancy occurs in the non-centered case $\E(\pi_{\star}\mu)\neq 0$, where we also need take into acount the recentering parameter. This forces to deal with martingales instead of an i.i.d. additive Markov chain on $\R^d$.  The complications brought up by this recentering can be observed in the formula for $A_{\nu}$. 

\bigskip

\noindent The paper is organized as follows. 

\smallskip

\Cref{Sec-deviation} explains how $\nu$-typical geodesic rays on $\Omega$ may be simulated with a small error by the $\mu$-sample paths on $\Omega$. More precisely, we recall various results of geodesic tracking from \cite{Tiozzo15, Benard21-asympt, Choi21-CLT}, and establish a uniform deviation inequality conditionally to a stopping time. 

\Cref{Sec-Norm&Busemann} recalls from \cite{BQ16-CLThyp} how  the distance function can be replaced by a cocycle in order to study  $\mu$-sample paths on $\Omega$. As a first application, we  estimate the time of exit of sample paths from a ball,  as well as the overshoot. 
 
\Cref{Sec-laws-projected} is dedicated to the proof of Theorems \ref{laws-projected}, \ref{laws-projected-weak}. It  relies on the two previous sections and on \Cref{Sec-general-laws}. We also comment on the limit covariance matrix $A_{\nu}$. 

\Cref{Sec-winding-hyperbolic} starts by recalling the notion of Patterson-Sullivan measures associated to a geometrically finite discrete subgroup of isometries of $\Hn$. We then prove winding statistics for PS-measures in the convex cocompact case (\Cref{laws-hyperbolic}), and establish singularity with harmonic measures in the presence of cusps (\Cref{singularity}). 
 
 \Cref{Sec-questions} records a few open questions leading to natural continuations of our study. 
 
\Cref{Sec-general-laws}  establishes limit laws for martingales arising from cocycles. The interest is that our setting is very general, encapsulating that of \cite{BQRW, BQ16-CLThyp, deAcosta83}, and deals with optimal\footnote{at least for the CLT, LIL, PLD.} moment assumptions. We also obtain a stronger version of the CLT, in which the time variable $n$ is replaced by a family of stopping times $(\tau_{n})$ approximating $n$ at scale $\sqrt{n}$.

\bigskip
\noindent{\bf Remark}. We believe the assumption of properness on $\Omega$ is not essential in the paper, and  could be  replaced by only assuming separability. Without properness, the Gromov boundary does not identify with  classes of geodesic rays, but with classes of $(1,Q)$-quasi geodesic rays, where  $Q>0$ is some fixed constant determined by the hyperbolicity constant of $\Omega$ \cite[Remark 2.16]{KapovichBenakli02}. Hence, we  rather need define $r_{\xi}$ as a $(1,Q)$-quasi geodesic ray from $o$ to $\xi$. To adapt the proof, we point out that we only use the properness assumption on $\Omega$ is  \Cref{Sec-Norm&Busemann}, while refering to Benoist-Quint's paper \cite{BQ16-CLThyp} to introduce the horofunction boundary, the Busemann cocycle, and its properties. We believe \cite{BQ16-CLThyp} could be adapted almost verbatim to remove the properness assumption, using the general theory of horofunction compactification developed by Maher and Tiozzo in \cite{MaherTiozzo18}. Asking that  $\Omega$ be separable would guarantee as in  \cite{MaherTiozzo18} that this compactification is indeed a metric space. Not to burden the article, we do not  pursue in this direction.

\bigskip

\noindent{\bf Acknowledgements}. I am thankful to P\'eter Varj\'u  for useful comments on a preliminary version of this text. I am also grateful to the anonymous referee for his meticulous reading and valuable remarks.

\section{Geodesic tracking of sample paths} \label{Sec-geod-app} \label{Sec-deviation}

We  explain how the distribution of $\nu$-typical geodesic rays on $\Omega$ can be approximated  by that of $\mu$-sample paths. 

\bigskip

We keep the notations ($\Omega$, $d_{\Omega}$, $o$, \details{$(r_{\xi})_{\xi\in \partial \Omega}$}, $\Gamma$) of \Cref{Sec-intro}, let $\mu$ be a non-elementary probability measure on $\Gamma$ and $\nu$ its Furstenberg measure on $\partial \Omega$.  We set $(B,\beta)=(\Gamma^{\N^*}, \mu^{\otimes \N^*})$ and recall that for $\beta$-typical $b\in B$, the sample path $(b_{1}\dots b_{n}.o)_{n\geq 0}$ on $\Omega$ converges to some point in the Gromov boundary $\xi_{b}\in \partial \Omega$, whose distribution under $\beta$ coincides with $\nu$. The corresponding geodesic ray $r_{\xi_{b}}$ is said to be \emph{asymptotic} to the sample path $(b_{1}\dots b_{n}.o)_{n\geq 0}$.   We also write $t_{b,n}=d_{\Omega}(o, b_{1}\dots b_{n}.o)$. 

\bigskip

We start with a very general result stating that $\mu$-sample paths on $\Omega$ have a small chance to deviate from their asymptotic geodesic ray, regardless of the strategy deciding the time of observation. A strategy is incarnated by a stopping time on $(B, \beta)$, implicitely refering  to   the filtration $(\mathcal{B}_{k})_{k\geq 1}$ where $\mathcal{B}_{k}$ is the sub $\sigma$-algebra of $B$ generated by $b_{1}, \dots, b_{k}$.  

\begin{thm}[Deviation inequality with a stopping time] \label{deviation} \details{Let $\mu$ be a non-elementary probability measure on $\Gamma$ and $\nu$ its Furstenberg measure on $\partial \Omega$. }
There exists a  family $(\eps_{R})_{R\geq0}\in (0, 1]^{\R^+}$ such that $\eps_{R}\to 0$ as $R\to+\infty$, and   
for every $\beta$-a.e. finite stopping time $\tau : B\rightarrow \N^*\cup\{\infty\}$, every  $R>0$, 
$$\beta\{ b: d_{\Omega}\big(b_{1}\dots b_{\tau(b)}.o , \,r_{\xi_{b}}(t_{b,\tau(b)})\big)>R\}\leq \eps_{R} $$
Moreover, if $\mu$ has finite exponential moment, then one can choose  $\eps_{R}=D e^{-\delta R} $ for some constants $D, \delta>0$.
 
\end{thm}

\noindent{\bf Remarks.}
1) For $\tau=n$ constant, this result follows from  \cite[Proposition 2.11]{BoulMathSertSisto21} and the remark that $d(b_{1}\dots b_{n}, r_{\xi_{b}}(t_{b,n})) \leq  2d(b_{1}\dots b_{n}, r_{\xi_{b}}(\R^+))$. We explained earlier in the strategy of proof why we need the stronger version above.

2) If $\mu$ has finite $p$-th moment where $p>1$, the proof of \Cref{deviation} combined with the $\log^{p-1}$-regularity of $\nu$ \cite[Proposition 4.2]{BQ16-CLThyp} allows to choose $\eps_{R}$ so that $\int_{R>0}\eps_{R}R^{p-2} dR<\infty$.

\bigskip

\begin{proof}[Proof of \Cref{deviation}]
For $x\in \Omega$, $\xi \in \partial \Omega$, we let $F(x, \xi)$ be the image of any geodesic ray from $x$ to $\xi$. By hyperbolicity, the family $(F(x,\xi))_{x\in \Omega, \xi \in \partial \Omega}$ is coarsely $\Gamma$-equivariant, meaning that for  $ g\in \Gamma$, the distance between 
$g.F(x,\xi)$ and   $F(gx, g\xi)$ is bounded by $O(\delta_{\Omega})$ where $\delta_{\Omega}$ is a hyperbolicity constant of $\Omega$. 

\details{
Observe that for any $s\in \R^+$,
\begin{align*}
d_{\Omega}(b_{1}\dots b_{\tau(b)}.o, \, r_{\xi_{b}}(t_{b,\tau(b)})) & \leq d_{\Omega}(b_{1}\dots b_{\tau(b)}.o, \, r_{\xi_{b}}(s)) +d_{\Omega}(r_{\xi_{b}}(s), \, r_{\xi_{b}}(t_{b, \tau(b)})) \\
& = d_{\Omega}(b_{1}\dots b_{\tau(b)}.o, \, r_{\xi_{b}}(s)) +|s-t_{b, \tau(b)}|\\
 &\leq 2d_{\Omega}(b_{1}\dots b_{\tau(b)}.o, \, r_{\xi_{b}}(s))
\end{align*}
where the last line uses that  $| d_{\Omega}(x,z)-d_{\Omega}(y,z)|\leq d_{\Omega}(x,y)$ for all $x,y,z\in \Omega$.
}

 Taking the infimum over $s\in \R^+$, we deduce
\begin{align*}
\frac{1}{2}d_{\Omega}(b_{1}\dots b_{\tau(b)}.o, \, r_{\xi_{b}}(t_{b,\tau(b)})) 
&\leq  d_{\Omega}(b_{1}\dots b_{\tau(b)}.o, F(o, \xi_{b})) + O(\delta_{\Omega})\\
&=d_{\Omega}(o, F(b^{-1}_{\tau(b)}\dots b^{-1}_{1}.o, \xi_{T^{\tau(b)}b})) + O(\delta_{\Omega})\\ 
&\leq (b_{\tau(b)}^{-1}\dots b_{1}^{-1} .o\,|\, \xi_{T^{\tau(b)}b})_{o} + O(\delta_{\Omega})
\end{align*}
where $T: B\rightarrow B, (b_{i})_{i\geq1}\mapsto (b_{i+1})_{i\geq1}$ is the one-sided shift, $(.\,|\,.)_{o}$ denotes the Gromov product based at the origin $o$. \details{For the last inequality we use the fact that for $x,y\in \Omega$, $(x|y)_{o}\geq d_{\Omega}(o, [x,y]) - O(\delta_{\Omega})$ where $[x,y]$ is the geodesic segment from $x$ to $y$ (see for instance \cite[Chapter 2 \S 1.4]{Ghys83}) and pass to the limit to allow $y\in \Omega \cup \partial \Omega$.} 

In particular, using that $T^{\tau(b)}b$ varies with law $\beta$ when $b$ varies with law $\beta$, and does so independently from $b_{1},\dots b_{\tau(b)}$, we get

\begin{align*}
\beta\{ b : d_{\Omega}(b_{1}\dots b_{\tau(b)}.o ,\, r_{\xi_{b}}(t_{b,\tau(b)}))>2R\}& \leq \beta^{\otimes 2}\{ (b,a): (b_{\tau(b)}^{-1}\dots b_{1}^{-1}.o \,|\, \xi_{a})_{o}>R-O(\delta_{\Omega})\}\\
&\leq \sup_{x\in \Omega} \beta \{b : (x \,|\, \xi_{b})_{o}>R-O(\delta_{\Omega})\}
\end{align*}

Given $x\in \Omega$, $\xi, \eta\in \partial \Omega$, the definition of the hyperbolicity of $\Omega$ in terms of Gromov product \details{\cite[Chapter 2 \S 1.3]{Ghys83}} gives 
$$
 \left.
    \begin{array}{ll}
        (x| \xi)_{o}\geq R & \\
        (x|\eta)_{o}\geq R &
    \end{array}
\right \} \implies (\xi|\eta)_{o}\geq R -O(\delta_{\Omega})
$$ 
\details{Endow $\partial \Omega$ with a Gromov metric $d_{\partial_{\Omega}}$ \cite[Chapter 7 \S 3]{Ghys83}. It satisfies that for some small enough $\eps>0$, one has $d_{\partial_{\Omega}}(\xi, \eta)\leq \eps^{-1}e^{-\eps (\xi|\eta)_{o}}$.  In particular, the set $\{\xi \in \partial \Omega:  (x \,|\, \xi)_{o}>R-O(\delta_{\Omega})\}$ is included in a ball of $\partial \Omega$ whose radius is exponentially small in $R$.}  The fact that $\nu$ has no atom, and is even H\"older regular when $\mu$ has finite exponential moment \cite[Corollary 2.13]{BoulMathSertSisto21}, finishes the proof. 

\end{proof}

\bigskip

We now recall various bounds for the distance separating a fixed typical sample path on $\Omega$ and its asymptotic geodesic ray. 
\begin{thm}[Geodesic tracking] \label{geod-track} 
 \details{Let $\mu$ be a non-elementary probability measure on $\Gamma$ and $\nu$ its Furstenberg measure on $\partial \Omega$. }
\begin{itemize}

\item For every $\eps>0$, there exists  $R>0$ such that for $\beta$-almost every $b\in B$, 
$$\liminf_{n\to+\infty} \, \frac{1}{n}\, \sharp\{i\leq n \,:\,d_{\Omega}(b_{1}\dots b_{i}.o, r_{\xi_{b}}(t_{b,i}))\leq R\} \geq 1- \eps$$

\item Assume $\mu$ has finite $p$-th moment where $p>0$. Then for $\beta$-almost every $b\in B$, 
$$d_{\Omega}(b_{1}\dots b_{n}.o, r_{\xi_{b}}(t_{b,n}))=o(n^{\frac{1}{2p}})$$

\item
Assume $\mu$ has finite exponential moment. There exists a constant $D>0$ such that for $\beta$-almost every $b\in B$, for large $n\geq0$, 
$$d_{\Omega}(b_{1}\dots b_{n}.o , r_{\xi_{b}}(t_{b,n}))\leq D\log n$$
\end{itemize}

\end{thm}

\bigskip

By way of proof, we make a few comments. 

\bigskip
\noindent{\bf Remarks}.
1) Notice that the first bound does not require any moment condition on $\mu$. A similar result in the context of random walks on homogeneous spaces of arbitrary rank appears in \cite[Theorem A bis]{Benard21-asympt} and the proof easily extends to our setting.  We will not use this estimate, but we recall it for completeness.

2) For $\mu$ with finite first moment,  Tiozzo establishes sublinear escape \cite{Tiozzo15}.  For any $p>0$, the stronger estimates above is proven by Choi in \cite[Theorem D]{Choi21-CLT}. 

3) The  third bound assuming finite exponential moment  follows directly from  \Cref{deviation} and the Borel-Cantelli lemma.

\section{The norm and the Busemann cocycle} 

\label{Sec-Norm&Busemann}

The proof of Theorems \ref{laws-projected-weak},  \ref{laws-projected} requires a fine understanding of the distance $t_{b,n}$. However, it is not so handy to manipulate. We show here that $t_{b,n}$ can be  replaced by a cocycle, up to a bounded error. The latter is much easier to deal with.  A similar analysis is  carried out by Benoist-Quint in \cite[Sec. 2.3, 3.2, 4.3]{BQ16-CLThyp}. As a first application, we infer estimates for the exit time and overshoot of $t_{b,n}$ from a bounded interval.

We let $(\Omega, d_{\Omega}, o, \Gamma)$ be as in \Cref{Sec-intro}, and fix $\mu$ a non-elementary probability measure on $\Gamma$.

\subsection{Busemann cocycle} \label{Sec-bus-coc}

Every $z\in \Omega$ defines a so-called horofunction 
$$h_{z}: \Omega \rightarrow \R, \,y\mapsto d_{\Omega}(z,y)-d_{\Omega}(z,o)$$
and the map $\Omega \rightarrow C^0(\Omega, \R)$ is   injective continuous when $C^0(\Omega, \R)$ is given the (metrizable) compact-open topology. The \emph{Busemann compactification} $\GGb$ of $\GG$ is the closure of $\{h_{z}, z\in \GG\}$ in $C^0(\GG, \R)$. The injection $z\mapsto h_{z}$ is an embedding that identifies $\GG$ with a dense open set of $\GGb$. Moreover,  as every horofunction is $1$-Lipschitz and vanishes at the  basepoint,  $\GGb$ must be compact.  We call \emph{Busemann boundary} of $\GG$ the set $X= \GGb\smallsetminus \GG$. The $\Gamma$-action on $\GG$ extends continuously to the boundary, via the formula: $\forall x\in \GGb$, $z\in \Omega$, 
$$h_{\gamma.x}(z)=h_{x}(\gamma^{-1}z)- h_{x}(\gamma^{-1}.o) $$ 

\bigskip
\noindent{\bf Remark.} At the contrary of the Gromov boundary $\partial \Omega$, the Busemann boundary $X$ is not invariant by quasi-isometry.  For instance, let $\Gamma=F_{2}\times \Z/2\Z$ where $F_{2}$ is the group freely generated by two elements $\{u,v\}$. For the systems of generators $S_{1}= \{e, u^{\pm 1}, v^{\pm 1}\}\times \{0,1\}$, $S_{2}= \{e, u^{\pm 1}, v^{\pm 1}\}\times \{1\}$, the Busemann boundaries of the associated Cayley graphs are\footnote{To see this, call $\GG_{1}, \GG_{2}, \mathcal{H}$ the Cayley graphs of $(\Gamma, S_{1})$, $(\Gamma, S_{2})$,  $(F_{2}, \{u^{\pm1}, v^{\pm1}\})$, and $\partial_{B}\GG_{1}, \partial_{B}\GG_{2}, \partial_{B}\mathcal{H}$ their Busemann boundaries. They come with  natural projections $\GG_{i}\rightarrow \mathcal{H}, x\mapsto \overline{x}$. On the one hand, we   have $h^{\GG_{1}}_{x}(y)=  h^{\mathcal{H}}_{\x}(\y)$ as long as $d_{\mathcal{H}}(\x,\y)\geq 1$, so $\partial_{B} \GG_{1} =\partial_{B} \HH= \partial F_{2}$. However, for every $p\in F_{2}$, the horofunctions $h^{\GG_{2}}_{(p,0)}$ and $h^{\GG_{2}}_{(p,1)}$ do not have the same parity on $\Gamma$, hence the distance between them is bounded below positively independently of $p$. This implies that any sequence converging to the Busemann boundary eventually remains in one leaf $F_{2}\times \{*\}$, and the  equality $\partial \GG_{2} =\partial F_{2}\times \{0,1\}$ follows.
} respectively $\partial F_{2}$, $\partial F_{2}\times \{0,1\}$.

However, the identity map $\GG\rightarrow \GG$ extends continuously to a surjective $\Gamma$-equivariant map $X \rightarrow \partial \GG$.

\bigskip
We define the \emph{Busemann cocycle} $\sigma : \Gamma\times \GGb \rightarrow \R, (\gamma,x)\mapsto h_{x}(\gamma^{-1}.o)$. The term cocycle means that $\sigma$ satisfies the relation $\sigma(\gamma_{2} \gamma_{1},x)=\sigma(\gamma_{2},\gamma_{1}x)+\sigma(\gamma_{1},x)$ for $\gamma_{1},\gamma_{2}\in \Gamma$, $x\in \GGb$. It is also useful to observe that for all $\gamma\in \Gamma$, 
\begin{align}\label{major-coc}
\sup_{x\in \GGb}|\sigma(\gamma,x)|\leq d_{\Omega}(o,\gamma.o)
\end{align}
as  the $h_{x}$ are $1$-Lipschitz and vanish at $o$. In particular, if $\mu$ has finite $k$-th moment, then the family of variables $(\sigma(. ,x))_{x\in  \GGb }$ is uniformly bounded in $L^k(\Gamma, \mu)$. The next proposition claims that the cocycle $\sigma$  gives a good approximation of the distance to $o$ along the $\mu$-trajectories. It assumes no moment condition  on $\mu$.

\begin{prop}[Cocycle approximation] \label{cocycle-approx}
For every $\eps>0$, there exists $R>0$ such that for all $x\in X$, 
$$\beta \left(b \,:\, \sup_{n\geq 1} |\sigma(b_{n}^{-1}\dots b_{1}^{-1}, x)- t_{b,n} |  \leq R   \right) >1- \eps$$
\end{prop}

\begin{proof} This is  \cite[Proposition 3.2]{BQ16-CLThyp} applied to $(\Omega, d_{\Omega})$ and the image \details{$\widecheck{\mu}$} of $\mu$ by the inverse map $\gamma\mapsto \gamma^{-1}$.
\end{proof}

When $\mu$ has a finite second moment, the restriction of $\sigma$ to the boundary is known to be cohomologous to a cocycle with constant $\widecheck{\mu}$-drift \cite[Proposition 4.6]{BQ16-CLThyp}: there exist a cocycle $\sigma_{0} : \Gamma \times X \rightarrow \R$ and a \emph{bounded}  function $\psi : X \rightarrow \R$ such that 
  \begin{itemize}
     \item $\forall \gamma\in \Gamma$, $\forall x\in X$,  $$\sigma(\gamma,x)=\sigma_{0}(\gamma,x)+\psi(\gamma x)-\psi(x)$$

\item $ \forall x\in X$, $$\sum_{\gamma\in \Gamma} \sigma_{0}(\gamma^{-1},x) \mu(\gamma)=\lambda_{\mu}$$
  \end{itemize}

It is worth noting that the definition of $\psi$ is explicit: letting $\nu_{X}$ be any $\mu$-stationary probability measure on $X$, we can set 
$$ \psi(x)=\int_{X} (x \,|\,y)_{o}\,d\nu_{X}(y)$$  
where $(.\,|.)_{o}$ refers to the Gromov product extended to $X$ (see \cite[Section 2.3]{BQ16-CLThyp}).

\bigskip
\noindent{\bf Remark}. 
It will be important for us to know that $\psi$ is not only bounded, but \emph{continuous}. This property is not mentioned in \cite{BQ16-CLThyp}. However it can be deduced from the above formula, the fact that  $(.\,|.)_{o}$ is continuous, and Inequality (4.10) in \cite{BQ16-CLThyp}. Indeed the latter implies that the sub-integral $\sup_{x\in X}\int_{ \{(x \,|\,.)_{o}\geq R\}} (x \,|\,y)_{o}\,d\nu_{X}(y)=o(1)$ as $R\to +\infty$.

\subsection{Exit estimates for the norm} \label{Sec-stopping-times} \label{Sec-exit-norm}

Given $s>0$, set 
$$\tau_{s} : B\rightarrow \N\cup \{\infty\}, \,b\mapsto \inf \{k \geq 1,\, t_{b,k}\geq s \lambda_{\mu}\}.$$ 

As a first application of the cocycle approximation from \Cref{Sec-bus-coc}, we show  that  the overshoot $t_{b,\tau_{s}(b)}-s\lambda_{\mu}$ is essentially  bounded, and that $\tau_{s}$ lands roughly in a window of size $\sqrt{s}$ around $s$.

\begin{prop} \label{control-tau}
\begin{enumerate}

\item  \emph{(Overshoot control)} Assume $\mu$ has finite first moment. Then 
$$\sup_{s>0}\beta( b\,:\, t_{b,\tau_{s}(b)} >s\lambda_{\mu} +R ) \underset{R\to+\infty}{\longrightarrow} 0$$

\item  \emph{(Time control)} Assume $\mu$ has finite second moment. We then have: \break $\forall \eps>0$, $\exists R>0$, $\forall s> 1$, 
$$\beta( b\,:\,  |\tau_{s}(b) -s| > R\sqrt{s} ) \leq \eps$$

\end{enumerate}
\end{prop}

\begin{proof}
Fix $x\in X$. According to the cocycle approximation from \Cref{cocycle-approx}, we do not loose generality  by replacing $t_{b,n}$ by $\sigma(b^{-1}_{n}\dots b^{-1}_{1}, x)$, both in the definition of $\tau_{s}$ and in the  estimates of the proposition. The time estimate  for $\sigma(b^{-1}_{n}\dots b^{-1}_{1}, x)$ follows from  \Cref{general-tau} applied to $\lambda_{\mu}^{-1}\sigma_{0}$. The overshoot control  also follows from  \Cref{general-tau} after replacing $\mu$ by a fixed convolution power $\mu^{*n_{0}}$ to reduce to the case where the drift  has a  positive lower bound uniform on $X$. \details{Indeed \Cref{general-tau} is applicable for $\mu^{*n_{0}}$ if  $n_{0}$ is large:  for any $n_{0}\geq 1$, the family $(\sigma(.,x))_{x\in X}$ is uniformly bounded in $L^1(\Gamma, \mu^{*n_{0}})$ and for large $n_{0}$,  for all $x\in X$
$$\int_{\Gamma} \sigma(\gamma^{-1}, x) \,d\mu^{*n_{0}}(\gamma) \geq \frac{1}{2} n_{0}\lambda_{\mu}$$
due to \Cref{cocycle-approx}, the convergence  $n^{-1}t_{b,n}\rightarrow \lambda_{\mu}$ in $L^1(\Gamma^{\N^*}, \mu^{\otimes \N^*})$, and the domination $\sup_{x\in X} |\sigma(\gamma, x)|\leq d_{\Omega}(o,\gamma.o)$. }
\end{proof}

\section{Limit laws for  geodesic rays} \label{Sec-laws-projected}

This section is dedicated to the proof of \Cref{laws-projected} and \Cref{laws-projected-weak}. Several parts of the proof can be extended to general limit theorems for cocycles. Those are postponed to \Cref{Sec-general-laws}. 

\subsection{Law of large numbers}

We establish the law of large numbers  for  $\nu$-typical geodesic rays on $\Omega$ projected to $\R^d$. We argue under various  conditions of regularity on the projection $i$.

\bigskip
For  $i$  quasi-Lipschitz, we prove the strong law of large numbers in \Cref{laws-projected}. The proof uses the sublinear geodesic tracking from \Cref{geod-track}.
\begin{proof}[Proof of \Cref{laws-projected} (LLN)]

By \Cref{geod-track} and the quasi-Lipschitz assumption on $i$, we have that for $\beta$-almost every $b\in B$, $n\geq 2$, 
$$\|i (b_{1}\dots b_{n}.o)-i\circ r_{\xi_{b}}(t_{b,n})\| =o(  n)  $$
or in other terms
$$\|\pi (b_{1}\dots b_{n})-i\circ r_{\xi_{b}}(t_{b,n})\| =o(  n)  $$
as $i (b_{1}\dots b_{n}.o)=\pi(b_{1}\dots b_{n})+i(o)$.
\details{Note that the assumption of finite first moment on $\mu$ for a word metric implies that its projection $\pi_{\star}\mu=\mu_{ab}$ has a finite first moment as  a measure on $\R^d$.} By the law of large numbers for the $\mu_{ab}$-walk on $\R^d$ and the definition of the rate of escape $\lambda_{\mu}$, we  know that 
$$\pi (b_{1}\dots b_{n})=n \E(\mu_{ab}) +o(n)\,\,\,\,\,\,\,\,\,\,\,\,\,\,\,\,\,\,\,\,\,\,\,\, t_{b,n} =n\lambda_{\mu} +o(n)  $$
Combining these estimates, we get
$$\|n \E(\mu_{ab})-i\circ r_{\xi_{b}}(n\lambda_{\mu})\| = o(n)  $$
which implies that $i\circ r_{\xi_{b}}(n\lambda_{\mu}) = n\lambda_{\mu} e_{\nu} +o(n)$. The result follows for continuous times $t\in \R^+$ as $i$ is quasi-Lipschitz. 

\end{proof}

For $i$ locally bounded, we prove the weak law of large numbers in \Cref{laws-projected-weak}. The proof relies on the deviation inequality with stopping time from \Cref{deviation} and on the overshoot control in \Cref{control-tau}. 

\begin{proof}[Proof of \Cref{laws-projected-weak} (WLLN)]
For $n\geq 0$, recall the stopping time
$$\tau_{n} : B\rightarrow \N\cup \{\infty\}, \,b\mapsto \inf \{k \geq 1,\, t_{b,k}\geq n \lambda_{\mu}\}.$$ 
The deviation estimate  from \Cref{deviation} guarantees that for every $\eps>0$, there exists $R>0$ such that for all $n\geq 0$, 
$$\beta\{b\,:\, d(b_{1}\dots b_{\tau_{n}(b)}.o, r_{\xi_{b}}(t_{b,\tau_{n}(b)}))\leq R  \}\geq 1-\eps $$
Using \Cref{control-tau}, we may increase $R$ to also have for every $n\geq 0$, 
 $$\beta\{b\,:\, d(b_{1}\dots b_{\tau_{n}(b)}.o, r_{\xi_{b}}(n\lambda_{\mu}))\leq R  \}\geq 1-\eps $$

Using that $i$ is locally bounded and $\Gamma$-equivariant, we can see that those $b$ satisfy  $\|\pi(b_{1}\dots b_{\tau_{n}(b)})-i\circ r_{\xi_{b}}(n\lambda_{\mu})\|\leq C_{R}$ for some constant $C_{R}$ only depending on $(i,R)$.  
Hence, it is enough to show the almost-sure convergence
\begin{equation}
 \frac{\pi(b_{1}\dots b_{\tau_{n}(b)})}{n}\to \lambda_{\mu}e_{\nu}=\E(\mu_{ab}) \label{WLLN}
 \end{equation}
The equivalence $t_{b,n}\sim n\lambda_{\mu}$ implies $\tau_{n}\sim n$. Hence \eqref{WLLN} follows from the law of large numbers for $\mu_{ab}$ on $\R^d$. 
\end{proof}

\subsection{Central limit theorem} \label{Sec-proof-CLT}

We prove the central limit theorem from \Cref{laws-projected-weak}.

\begin{proof}[Proof of \Cref{laws-projected-weak} (CLT)]

We want to show that, for  $b$ varying with law $\beta$,  the family of random variables 
\begin{align} \label{CLT1}
\frac{i \circ r_{\xi_{b}}(s \lambda_{\mu}) - s \lambda_{\mu}e_{\nu}}{\sqrt{s}} 
\end{align}
converges to a (possibly degenerate) Gaussian distribution as $s\to+\infty$. By \Cref{deviation} and \Cref{control-tau} (overshoot control), for any $\eps>0$, there exists $R>0$ such that for any $s>0$, 
$$\beta(b\,:\,d(b_{1}\dots b_{\tau_{s}(b)}.o, r_{\xi_{b}}(s\lambda_{\mu}))  \leq R) \geq 1-\eps$$
Using that $i$ is locally bounded and $\Gamma$-equivariant, we deduce that, up to increasing $R$,  for any $s>0$, 
$$\beta(b\,:\,|\pi(b_{1}\dots b_{\tau_{s}(b)}) -i\circ r_{\xi_{b}}(s\lambda_{\mu})|   \leq R) \geq 1-\eps$$
Hence,   the CLT for $\nu$-typical geodesic rays  in  \eqref{CLT1} amounts to that of 
\begin{align} \label{CLT2}
\frac{\pi(b_{1}\dots b_{\tau_{s}(b)}) - s \lambda_{\mu}e_{\nu}}{\sqrt{s}} 
\end{align}
Fix $x\in \BG$. By \Cref{control-tau} (overshoot control) and \Cref{cocycle-approx},  the family of variables 
$$\frac{s\lambda_{\mu} - \sigma(b^{-1}_{\tau_{s}(b)} \dots b^{-1}_{1}, x )}{\sqrt{s}}$$
  converges to zero in probability as $s\to+\infty$. 
  Hence we just need to check the CLT for 
\begin{align} \label{CLT2}
\frac{\pi(b_{1}\dots b_{\tau_{s}(b)}) -  \sigma(b^{-1}_{\tau_{s}(b)} \dots b^{-1}_{1}, x )e_{\nu}}{\sqrt{s}}.
\end{align}
We introduce the cocycle $\tsigma : \Gamma\times \BG\rightarrow \R^d$ defined by 
$$\tsigma(g,x)= \pi(g)+\sigma_{0}(g,x)e_{\nu} $$
 This cocycle satisfies $\|\sigma(g,x)\|\ll |g|'$ for $|g|'=\|\pi(g)\|+d(o,g.o)$ and has constant zero drift for $\widecheck{\mu}$. It is also continous in the $\BG$-variable because $\sigma$ is continous and the function $\psi$ relating $\sigma$ and $\sigma_{0}$ is continuous. Finally \Cref{cocycle-approx} confirms that $\sigma$ is $(\widecheck{\mu}, \BG)$-controlled. These observations ensure we may apply our CLT with stopping time from the appendix (\Cref{general-CLT+}) to  $\tsigma$.  It tells us that   
 \begin{align*} 
\frac{i \circ r_{\xi_{b}}(s \lambda_{\mu}) - s \lambda_{\mu}e_{\nu}}{\sqrt{s}}  \underset{s\to+\infty}{\wconv} \mathscr{N}(0, \lambda_{\mu}A_{\nu})
\end{align*}
where $\lambda_{\mu} A_{\nu}$ is the limit, uniformly in $x\in \BG$:
 \begin{align*} 
\lambda_{\mu} A_{\nu}= \lim \frac{1}{n} \int_{G} \tsigma(g,x) \,{^t \tsigma(g,x)} \,d\widecheck{\mu}^{*n}(g)
 \end{align*}
Using the cocycle relation and that $\tsigma$ has zero drift for $\widecheck{\mu}$, we can replace $\sigma_{0}$  in the above limit  by the original Busemann cocycle $\sigma$:
 \begin{align*} 
\lambda_{\mu}  A_{\nu}= \lim \frac{1}{n} \int_{G} (\pi(g)+\sigma(g,x)e_{\nu})\,\, {^t(\pi(g)+\sigma(g,x)e_{\nu})} \,\,d\widecheck{\mu}^{*n}(g)
 \end{align*}
 This limit is still uniform in $x\in \BG$. 

   \end{proof}

  \subsection{The covariance matrix}
  
 We comment on the covariance matrix $A_{\nu}$  of the limit distribution in the CLT. 
  
  \bigskip
  
  We start with   a  criterion of non-degeneracy for $A_{\nu}$.  Recall that the \emph{stable length} of an element $\gamma\in \Gamma$ is the quantity $l(\gamma)=\lim_{n\to +\infty} \frac{1}{n}d_{\Omega}(o,\gamma^n.o)$. It is well defined by subadditivity. One can see it as an analog of the logarithm of the spectral radius for matrices in $SL_{d}(\R)$. 
  
  \begin{lemma}[Non-degeneracy criterion] \label{ND-criterion}
   If $\det(A_{\nu})= 0 $, then   
  $\langle \supp \mu_{ab} \rangle_{\text{Vect}}$ is a \emph{proper} subspace of $\R^d$ or there exists $\varphi \in (\R^d)^*$ such that $\varphi \circ \pi(g)=l(g)$ on   the semigroup generated by the support of $\mu$. 
  \end{lemma}

  \begin{proof}
  Assume $\det(A_{\nu})= 0 $ and $\langle \supp \mu_{ab} \rangle_{\text{Vect}} =\R^d$. By the first hypothesis, there exists a non zero linear form $\varphi$ such that $\varphi A_{\nu}=0$. By Remark 2) following \Cref{general-CLT}, we can rewrite for every $n\geq 1$, any $\widecheck{\mu}$-stationary measure $\nu_{\BG}$ on $\BG$:
  $$\lambda_{\mu}A_{\nu}= \frac{1}{n}\int_{\Gamma\times \BG} (\pi(\gamma)+\sigma_{0}(\gamma,x)e_{\nu}) \,{^t}(\pi(\gamma)+\sigma_{0}(\gamma,x)e_{\nu})\,d{\widecheck{\mu}}^{*n}(\gamma)d\nu_{\BG}(x).$$
Hence,  $\varphi A_{\nu} {^t\varphi}=0$ reads as 
 $$\int_{\Gamma \times \BG} \left(\varphi\circ   \pi(\gamma)+\sigma_{0}(\gamma,x) \varphi(e_{\nu})   \right)^2 d\widecheck{\mu}^{*n}(\gamma) \,d\nu_{\BG}(x)=0.  $$
Call $\Gamma_{\widecheck{\mu}}\subseteq \Gamma$ the semigroup  generated by the support of $\widecheck{\mu}$. We deduce that for $\nu_{\BG}$-almost every $x\in \BG$, for all $\gamma\in \Gamma_{\widecheck{\mu}}$, we have  
$$\varphi\circ   \pi(\gamma) =-\sigma_{0}(\gamma,x) \varphi(e_{\nu}).  $$
As  $\pi(\Gamma_{\widecheck{\mu}})$ spans $\R^d$ by assumption and $\varphi\neq 0$, we must have $\varphi(e_{\nu})\neq 0$, so we may assume $\varphi(e_{\nu})=1$. Using that the difference  $|\sigma-\sigma_{0}|$ is bounded, we infer that  for  $\nu_{\BG}$-almost every $x$, for $\gamma\in \Gamma_{\widecheck{\mu}}$, $n\geq 1$, 
 $$\varphi \circ \pi(\gamma)=\frac{1}{n}\varphi \circ \pi(\gamma^n)= -\frac{1}{n}\sigma_{0}(\gamma^n,x)=-\frac{1}{n}\sigma(\gamma^n,x)+o(1).$$ 
 \details{As the projection of $\nu_{\BG}$  to the Gromov boundary $\partial \Omega$ is the Furstenberg measure of $\mu$ and the latter has no atom, there exist in particular $x, y\in \BG$ with distinct projections on $\partial \Omega$ such that for all $\gamma\in \Gamma_{\widecheck{\mu}}$, $n\geq 1$, 
 $$\varphi \circ \pi(\gamma)= -\frac{1}{n} \max [\sigma(\gamma^n,x), \sigma(\gamma^n,y)]+o(1).$$ 
 However, given such $x,y$, there exists $C>0$ such that 
 $$ d_{\Omega}(o,g.o)\geq \max(\sigma(g,x), \sigma(g,y)) \geq d_{\Omega}(o,g.o)- C$$ 
 for all $g$ \cite[(4.3)]{BQ16-CLThyp}. Letting $n$ go to infinity, we get $\varphi \circ \pi(\gamma)=-l(\gamma)$.
As $l(\gamma^{-1})=l(\gamma)$ and $-\varphi \circ \pi(\gamma)=\varphi \circ \pi(\gamma^{-1})$, the result follows. 
}

  \end{proof}

  We also give an example showing that the rank of $A_{\nu}$ may be strictly greater  than the  rank of the covariance matrix of $\mu_{ab}$.  
  
  \begin{lemma} \label{exampleAnu}
 Let  $\Gamma$ be the free group with two generators $u,v$, write $\Omega$ the Cayley graph associated to $\{u,v\}$, and let $i: \Omega\rightarrow \R^2$ continuous $\Gamma$-equivariant such that $i(u)=(1,0)$, $i(v)=(0,1)$. For  $\mu=\frac{1}{3}(\delta_{u}+\delta_{uv}+\delta_{uv^{-1}})$ we have
   $$\det(Cov(\mu_{ab}))=0 \neq \det(A_{\nu})$$ 
  \end{lemma}
  
  \begin{proof}
 We have $\det(Cov(\mu_{ab}))=0$ because $\mu_{ab}$ is supported on an affine hyperplane of $\R^2$. Assume by contradiction that $\det(A_{\nu})=0$. By  \Cref{ND-criterion}, there exists $\varphi \in (\R^2)^*$ such that $\varphi \circ \pi(g)=l(g)$ on the semigroup generated by the support of $\mu$. Writing $e_{1}=(1,0)$ and $e_{2}=(0,1)$, we get  the equations 
 $$\varphi(e_{1}) =1, \,\,\,\,\varphi(e_{1}+e_{2}) =2, \,\,\,\,\varphi(e_{1}-e_{2}) =2.$$
 These are not compatible, whence the contradiction. 
 
  \end{proof}
  
  \subsection{Law of the iterated logarithm}

  We prove the law of the iterated logarithm  in  \Cref{laws-projected}. 
  
  \begin{proof}[Proof of \Cref{laws-projected} (LIL)]
  According to \Cref{geod-track} and the second moment assumption on $\mu$, one has for $\beta$-almost every $b\in B$, 
  $$d_{\Omega}(b_{1}\dots b_{n}.o,\, r_{\xi_{b}}(t_{b,n}))=o(\sqrt{n}) $$
 By the quasi-Lipschitz assumption on $i$, we infer
    $$  i\circ r_{\xi_{b}}(t_{b,n}) = \pi(b_{1}\dots b_{n}) +o(\sqrt{n}) $$

On the other hand, \Cref{cocycle-approx} yields that for a fixed $x\in X$, for $\beta$-almost every $b\in B$, we have 
$$t_{b,n}e_{\nu} = \sigma(b_{n}^{-1}\dots b_{1}^{-1}, x)e_{\nu} + O(1) $$

Putting these together, we get
  \begin{align*}
 \frac{i\circ r_{\xi_{b}}(t_{b,n})-t_{b,n} e_{\nu}}{\sqrt{2 t_{b,n} \log \log t_{b,n}}} =  \frac{-\tsigma(b^{-1}_{n}\dots b^{-1}_{1}, x)+o(\sqrt{n})}{(1+o(1))\sqrt{2\lambda_{\mu}n \log \log n}}
  \end{align*}
  where $\tsigma$ is the cocycle introduced in \Cref{Sec-proof-CLT}. 
 \details{We checked earlier in \Cref{Sec-proof-CLT} that $\tsigma$ satisfies the conditions of application of the general LIL for cocycles from \Cref{general-LIL}}. By this LIL, for $\beta$-almost every $b\in B$, 
    \begin{align}  \label{LIL-ray} 
  \left\{\text{ Accumulation points of } \left(\frac{i \circ r_{\xi_{b}}(t) - te_{\nu}}{\sqrt{2t \log \log t}}\right)_{t\in \{t_{b,n}, n\in \N \} } \right\} = A^{1/2}_{\nu}(\overline{B}(0,1)) 
    \end{align}

   We just need to check there are no other accumulation points. As $\mu$ has finite second  moment, an application of Borel-Cantelli Lemma yields for $\beta$-almost every $b\in B$, 
$$|t_{b,n+1}-t_{b,n}| =o(\sqrt{n})$$ 
In particular, for large $n$, 
$$|t_{b,n+1}-t_{b,n}| \leq  \sqrt{ t_{b,n}}$$
It follows that for $b$ typical, the intervals $$I_{n}(b)=[t_{b,n}-  \sqrt{ t_{b,n}}, \,\,t_{b,n}+ \sqrt{ t_{b,n}}]$$ 
cover a translate of $\R_{+}$. Moreover, the condition that $i$ is quasi-Lischitz implies that for $t$ varying inside  $I_{n}(b)$, the quantity $\frac{i\circ r_{\xi_{b}}(t)-te_{\nu}}{\sqrt{2t\log \log t}}$ is constant up to an error that tends to $0$ as $n$ goes to infinity. This implies that all the accumulation points of  $(\frac{i\circ r_{\xi_{b}}(t)-te_{\nu}}{\sqrt{2t\log \log t}})_{t>2}$ are attainable along a sequence of times  among the $t_{b,n}$'s.  The LIL now follows from  \eqref{LIL-ray}.   

  \end{proof}

\subsection{Principle of large deviations}

We prove the principle of large deviations in \Cref{laws-projected}.

\begin{proof}[Proof of \Cref{laws-projected} (PLD)]
As $i$ is quasi-Lipschitz, it is enough  to prove that, for every $\eps>0$, there exist $D, \delta>0$ such that for all $n\geq 0$, 
$$\beta \{b : \|i \circ r_{\xi_{b}}(n\lambda_{\mu})- n \lambda_{\mu}e_{\nu} \|\geq  n\eps \}\leq De^{-\delta n} $$
But we have the bound 
\begin{align*}
&\|i \circ r_{\xi_{b}}(n\lambda_{\mu})- n \lambda_{\mu}e_{\nu} \| \\
& \leq \|i \circ r_{\xi_{b}}(n\lambda_{\mu})-i \circ r_{\xi_{b}}(t_{b,n})\|+ \|i \circ r_{\xi_{b}}(t_{b,n})-i(b_{1}\dots b_{n}.o)\| +\|i(b_{1}\dots b_{n}.o)- n \lambda_{\mu}e_{\nu} \|\\
&\ll |n\lambda_{\mu} - t_{b,n}| + d(r_{\xi_{b}}(t_{b,n}), b_{1}\dots b_{n}.o)  +\|\pi(b_{1}\dots b_{n})- n \lambda_{\mu}e_{\nu} \|  +1
\end{align*}
It is then enough to prove a PLD for each of the terms on the right-hand side. The first term is treated by the PLD for the distance to the origin   \cite[Theorem 1.1]{BoulMathSertSisto21} , 
the second by the exponential deviation estimate in \Cref{deviation} for $\tau=n$, and the third by the classical PLD applied to $\mu_{ab}$. This concludes the proof. 

\end{proof}

  \subsection{Gambler's ruin}

We  prove the gambler's ruin estimate in \Cref{laws-projected}. In particular $d=1$. 

 \begin{proof}[Proof of \Cref{laws-projected} (GR)]
Let $\eps>0$. It is enough to show that 
\begin{align} \label{liminfGR}
\liminf_{s\to +\infty} \beta \left\{b : (\pi\circ r_{\xi_{b}}(t) -t e_{\nu})_{t>0} \text{ meets  $[l_{s}, +\infty)$ before $(-\infty, -k_{s}]$}   \right\} \, > \, \frac{k}{k+l} -\eps
\end{align}
Indeed, the lower bound for the $\limsup$ follows by symmetry, and one concludes letting $\eps$ go to $0^+$.

Let us establish \eqref{liminfGR}. Fix a (large)  parameter $C>1$ to be specified below, and $x\in \BG$. Write $M_{n}(b)=-\tsigma(b_{n}^{-1}\dots b^{-1}_{1}, x)=\pi(b_{1}\dots b_{n}) - \sigma_{0}(b^{-1}_{n}\dots b^{-1}_{1}, x)e_{\nu}$ the martingale associated to $x$.  For $s>C$, let $E_{s}$ be the set of elements $b\in B$ such that 
\begin{itemize}
\item $\forall j\in \N^*$, $|M_{j}(b) - (\pi(b_{1}\dots b_{j})-t_{b,j}e_{\nu}) | \leq C $
\item $\forall j\leq s^3$, $| \pi(b_{1}\dots b_{j})- i \circ r_{\xi_{b}}(t_{b,j})|\leq C \log s$
\item $\forall j\leq s^3$, $|t_{b,j+1}-t_{b,j}|\leq C \log s$
\end{itemize}
According to  \Cref{cocycle-approx}, \Cref{geod-track}, and the exponential moment assumption on $\mu$, we can choose $C=C(\eps)>1$ large enough so that for every $s>C$, one has $$\beta(E_{s})>1-\eps$$ 
 We can also assume that $i$ is $C$-quasi-Lipschtiz.
Let $F_{s}$ be the set of $b\in B$ such that $M_{n}(b)$ meets  $[l_{s} +2C\log s, +\infty)$ before $(-\infty, -k_{s}+3C^2\log s+3C]$ and during the time interval $[0, s^3]$. Lemmas \ref{exit}, \ref{general-GR} from the appendix \details{can be applied thanks to \Cref{cocycle-approx} and yield
$$ \beta(F_{s})  \underset{s\to +\infty}{\longrightarrow} \frac{k}{k+l}$$}

Finally we  check that for $s>C$, $b\in E_{s}\cap F_{s}$, the recentered projected geodesic ray  $(\pi\circ r_{\xi_{b}}(t) -t e_{\nu})_{t> 0}$ meets  $[l_{s}, +\infty)$ before $(-\infty, -k_{s}]$. Indeed, let $j_{0}\leq s^3$ minimal such that $M_{j_{0}}(b) \geq l_{s}+2C\log s$. Using the first two items in the definition of $E_{s}$, we get that $\pi\circ r_{\xi_{b}}(t_{b,j_{0}})-t_{b,j_{0}}e_{\nu}\geq l_{s}$. 

We also check that $(i\circ r_{\xi_{b}}(t) -t e_{\nu})_{t> 0}$ does not meet $(-\infty, -k_{s}]$ for $t<t_{b,j_{0}}$. Suppose by contradiction it is the case: $\exists t_{1}<t_{b,j_{0}}$, 
$$i\circ r_{\xi_{b}}(t_{1}) -t_{1} e_{\nu}<-k_{s}$$
  Let  $0\leq j_{1}< j_{0}$ such that $t_{b,j_{1}}\leq t <t_{b,j_{1}+1}$ and observe that by the last item in the definition of $E_{s}$, we have $|t-t_{b,j_{1}}|\leq C\log s$. Using that $i: \GG\rightarrow \R^d$ is $C$-quasi-Lipschitz, we get 
  $$i\circ r_{\xi_{b}}(t_{b,j_{1}}) -t_{b,j_{1}}e_{\nu} \leq -ks+ 2C^2\log s+2C$$
  Then first two items in the definition of $E_{s}$  yield $M_{j_{1}}(b) \leq -k_{s} +3C^2\log s +3C$. Hence a contradiction.

\end{proof}

\section{Winding and Patterson-Sullivan measures}

\label{Sec-winding-hyperbolic}

The goal of this section is to prove Corollaries \ref{laws-hyperbolic}, \ref{singularity}. We first recall basic notions of hyperbolic geometry. 
\bigskip

Let $\Hn$  be the the $N$-dimensional real hyperbolic space, $\Gamma \leq Isom^+(\Hn)$ a non-elementary  discrete subgroup of orientation preserving isometries. \details{If the orbits of $\Gamma$ in $\Hn$ are not too sparse, one can associate to $\Gamma$ a certain class of measures on the visual boundary $\partial \Hn\simeq \mathbb{S}^{n-1}$ called Patterson-Sullivan measures. We  recall their definition, for a more general exposition see \cite{Roblin03}. }

The discrete group $\Gamma$ comes with  a \emph{limit set} $\Lambda_{\Gamma}\subseteq \partial\Hn$ and a \emph{critical exponent} $\delta_{\Gamma}\in \R^+$ defined by 
$$\Lambda_{\Gamma}=\overline{\Gamma.z}\cap \partial\Hn  \,\,\,\,\,\,\,\,\,\,\,\,\,\,\,\,\,\,\,\,\,\,\,\,\,\,\,\,\,\,\,\,\,\,\,\,\delta_{\Gamma} =\lim_{R\to+\infty} \frac{1}{R}\log \sharp B_{\Hn}(z,R)\cap \Gamma.z $$ 
 where $z\in \Hn$ is any point, $\overline{\Gamma.z}$ is the closure of $\Gamma.z$ in $\overline{\Hn}=\Hn\cup \partial \Hn$,  and $B_{\Hn}(z,R)$ denotes the ball of radius $R$ centered at $z$. 
As $\Gamma$ is non-elementary, we must have \details{$\sharp \Lambda_{\Gamma}=+\infty $}, $\delta_{\Gamma}>0$. \details{The number $\delta_{\Gamma}$ also corresponds to the exponent of convergence of the Poincaré series of $\Gamma$, that is the infimum of the positive numbers $s\in \R^*_{+}$ such that for some (hence any) $x,y\in \Hn$, one has
$$\mathcal{P}_{\Gamma}(s, x,y):=\sum_{\gamma \in \Gamma}e^{-s \,d_{\Hn}(x, \gamma.y)}<\infty$$
The value of $\mathcal{P}_{\Gamma}(s,x,y)$ at $s=\delta_{\Gamma}$ may be finite or infinite (this does not depend on the choice of $x,y$), and $\Gamma$ is said to be of \emph{divergence type} if it is infinite.  
}

We  also define the \emph{Busemann function}  by 
 $$\beta_{x}(z_{1},z_{2})=\lim_{z'\in \Hn, z'\to x}d(z',z_{1})-d(z',z_{2})$$
 for $x\in \partial \Hn$, $z_{1},z_{2}\in \Hn$.

\begin{def*}[Patterson-Sullivan measures \cite{Patterson76, Sullivan79}] Let $\Gamma \leq Isom^+(\Hn)$ be a non-elementary  discrete subgroup of orientation preserving isometries.
 There exists a  family of  finite Borel  measures $(\nu_{z})_{z\in \Hn}$ on $\Lambda_{\Gamma}$ such that for all $z,z'\in \Hn$, $\gamma\in \Gamma$, 
$$\gamma_{\star} \nu_{z}=\nu_{\gamma.z} \,\,\,\,\,\,\text{  and   }\,\,\,\,\,\, d\nu_{z}(\xi)=e^{-\delta_{\Gamma} \beta_{\xi}(z,z') }d\nu_{z'}(\xi)$$
\details{Moreover, if $\Gamma$ is of divergence type, then this family is unique up to scalar multiplication.}
In this case, we call the measures $(\nu_{z})_{z\in \Hn}$ the \emph{Patterson-Sullivan measures} associated to $\Gamma$. 
\end{def*}

\noindent{\bf Remark}. For $\Gamma$ of divergent type, the PS measures of $\Gamma$ are all equivalent to one another. Hence, they define without ambiguity a measure class.

\bigskip
\details{Let $\text{hull}(\Lambda_{\Gamma})$ be the convex hull of $\Lambda_{\Gamma}$, in other terms $\text{hull}(\Lambda_{\Gamma})$  is\footnote{\details{This characterization is indeed equivalent due Caratheodory's theorem applied to the projective model of $\Hn$, in which geodesics are straight lines.}} the set of points in $\overline{\Hn}$ belonging to an ideal polyhedral $n$-simplex with all vertices in $\Lambda_{\Gamma}$. In particular, $\text{hull}(\Lambda_{\Gamma})$ is compact because $\Lambda_{\Gamma}$ is.} We say that $\Gamma$  is \emph{convex cocompact} if the action of $\Gamma$ on $C(\Gamma)=\text{hull}(\Lambda_{\Gamma})\cap \Hn$  is cocompact. It is equivalent to ask that $\Gamma$ is finitely generated and the orbit maps $\theta_{z} :\Gamma\rightarrow \Hn, \gamma \mapsto \gamma.z$ for $z\in \Hn$ are quasi-isometric embeddings \details{(for some/any word metric on $\Gamma$), see for instance \cite[Section 1.8]{Bourdon95}. } 

Convex cocompact subgroups \details{are of divergence type}, and one can realise their PS measures on the boundary $\partial \Hn$  as Furstenberg measures for a suitable probability measure on $\Gamma$:

\begin{fact}[\cite{LiNaudPan21}] \label{fact-FL}
Let $\Gamma \leq Isom^+(\Hn)$ be a \emph{convex cocompact} non-elementary  discrete subgroup of orientation preserving isometries.
For any $z\in \Hn$, there exists a probability measure $\mu_{z}$ on $\Gamma$ whose support generates $\Gamma$ as a group, which has  finite exponential moment,  and  
such that $\mu_{z}*\nu_{z}=\nu_{z}$. 
\end{fact}

\subsection{Winding on convex cocompact manifolds}

\label{Sec-laws-hyperbolic}

We now prove \Cref{laws-hyperbolic}. We assume $\Gamma$ non-elementary, torsion free,   discrete, convex cocompact, and we  fix $m$ a probability measure on the boundary that is absolutely continuous with respect to the class of Patterson-Sullivan. We aim to estimate the statisics of $m$-typical geodesic rays on $\Hn$ projected via the map $i : \Hn\rightarrow H_{1}(M, \R)$ where    $M=\Gamma\backslash \Hn$. We can assume without loss of generality that the basepoint $o$ is in $C(\Gamma)$.  As $C(\Gamma)$ has compact projection, the norms of the $1$-forms defining $i$ are bounded in $C(\Gamma)$, hence \emph{$i$ is Lipschitz} on $C(\Gamma)$. We are thus  in a position to apply \Cref{laws-projected} \details{with $\Omega=C(\Gamma)$}  if $m$ is a Furstenberg measure for some probability measure $\mu$ with moment conditions. Note this is not always the case, for instance if $m$ lacks regularity. Our strategy is to start with the case where $m=\nu_{o}$ is the PS measure based at $o$, then deduce the general result using the $\Gamma$-equivariance of the $\{\nu_{z}, z\in \Hn\}$ to approximate any $m$ by a combination of $\nu_{\gamma.o}$.  

\bigskip

We start by proving that $\nu_{o}$-typical geodesic rays have zero drift in $H_{1}(M, \R)$. Our argument relies on the dynamics of the geodesic flow for the Bowen-Margulis-Sullivan measure (denoted by $m_{BMS}$) that we now recall.

The BMS-measure can be defined as the unique  probability measure on $T^1M$ which is invariant by the geodesic flow and has maximal entropy \cite{OtalPeigne04}.  It is supported on the compact  subset $T^1M_{|K(\Gamma)}$ where $K(\Gamma)=\Gamma\backslash C(\Gamma)$ is the \emph{convex core}. 

The BMS-measure also has an explicit construction in terms of Patterson-Sullivan measures. Indeed, recall Hopf coordinates allow to identify $T^1\Hn$ with $ (\partial\Hn)^{(2)} \times \R$ where $^{(2)}$ refers to the direct product minus the diagonal. The identification is given by the map 
$$ v \mapsto (v^-=\lim_{-\infty} \phi_{t}v, \,v^+=\lim_{+\infty} \phi_{t}v,\, B_{v^+}(o, v_{*}))$$
where $\phi_{t}$ is the time $t$ geodesic flow, $v_{*}\in \Hn$ stands for the basepoint of $v$. We define a measure $\widetilde{m}_{BMS}$ on $T^1\Hn$ by setting in Hopf coordinates   
$$d\widetilde{m}_{BMS}(x, y, s )= e^{\delta_{\Gamma} B_{\xi}(o, v_{*})}e^{\delta_{\Gamma} B_{\eta}(o, v_{*})} d\nu_{o}(x)d\nu_{o}(y) ds $$
where $v$ is here the vector represented by $(x, y, s )$. The measure $\widetilde{m}_{BMS}$ is $\phi_{t}$-invariant  (use that $\phi_{t}$ acts by translation on the $\R$-coordinate) and $\Gamma$-invariant.  Hence, it  defines a  $\phi_{t}$-invariant  measure on $T^1M$, which is  known to be finite and $\phi_{t}$-ergodic as $\Gamma$  is geometrically finite \cite{Sullivan84}. Its normalization of mass $1$ is $m_{BMS}$.

\bigskip

We now prove the centered law of large numbers in \Cref{laws-hyperbolic}. 

\begin{lemma}[zero drift] \label{zero-drift}
We have $$e_{\nu_{o}}=0$$  
\end{lemma}

\begin{proof}
 Extend $i$ to $T^1\Hn$ by  precomposing it with the natural projection $T^1\Hn\rightarrow \Hn$. 
 Introduce the \emph{cocycle}
$$\sigma : T^1M\times \R \rightarrow H_{1}(M, \R), \,\,(v, t)\mapsto i(\phi_{t}\widetilde{v})-i(\widetilde{v})$$
where $\widetilde{v}\in T^1\Hn$ lifts $v$.  
Note that the $\Gamma$-equivariance of $i$ guarantees  that $\sigma$ is well defined. One can interpret $\sigma(v,t)$ has the winding of the path $[v, \phi_{t}v]$ on $M$. Also, $\sigma$ is a cocycle in the sense that 
$$\sigma(v, s+t)=\sigma(v, s)+\sigma(\phi_{s}v, t) $$
Moreover, $\sigma(v, t)$ is uniformly bounded on the compact set $T^1M_{|K(\Gamma)}\times [-1, 1]$. 

The ergodicity of $(m_{BMS}, \phi_{t})$  combined with Birkhoff Ergodic Theorem yields that for $m_{BMS}$-almost-every $v_{0}\in T^1M$, 
$$\frac{1}{n} \left(i(\phi_{n}v_{0}) -i(v_{0})\right) =\frac{1}{n}\sum_{k=0}^{n-1} \sigma(\phi_{k}v_{0}, a_{1}) \longrightarrow \int_{T^1M}\sigma(v, 1) dm_{BMS}(v)$$

The integral on the right must be zero. Indeed, the measure $m_{BMS}$ is invariant under $\phi_{1}$ and the involution map $T^1M\rightarrow T^1M, v\mapsto \eps(v)$ reversing  orientation, so 
\begin{align*}
\int_{T^1M}\sigma(v, 1) dm_{BMS}(v)= \frac{1}{2}\int_{T^1M}\underbrace{\sigma(v, 1)+\sigma(\eps(\phi_{1}v), 1)}_{=0} dm_{BMS}(v)=0
\end{align*} 
\details{where the equality $\sigma(v, 1)+\sigma(\eps(\phi_{1}v), 1)=0$ comes from the computation $\sigma(\eps(\phi_{1}v), 1)=i(\phi_{1}\eps(\phi_{1}\widetilde{v}))-i(\eps(\phi_{1}\widetilde{v}))=i(\widetilde{v})-i(\phi_{1}\widetilde{v})=-\sigma(v, 1)$.}

It follows that for $\widetilde{m}_{BMS}$-almost every vector $\widetilde{v}\in T^1\Hn$, the path $[o, \phi_{t}\widetilde{v}]$ has sublinear drift in $H_{1}(M, \R)$.  As the projection of $\widetilde{m}_{BMS}$ to the boundary $\partial \Hn$ via $\widetilde{v}\mapsto \lim_{t\to+\infty}\phi_{t}\widetilde{v}$ is equivalent to $\nu_{o}$, we get nullity of the drift for $\nu_{o}$-typical directions, as required.

\end{proof}

We also claim that the Gaussian central limit theorem for $\nu_{o}$-typical rays is non-degenerate.

\begin{lemma}[definite covariance matrix]
We have $$\det(A_{\nu_{o}})\neq 0$$ 
\end{lemma}

\begin{proof}
\details{As $e_{\nu_{o}}=0$,   the fomula \eqref{formule-covariance} for the covariance matrix (established in \Cref{Sec-proof-CLT}) yields $A_{\nu_{o}}=Cov(\mu_{o,ab})$.} If $A_{\nu_{o}}$ were degenerate, then $\mu_{o,ab}$ would be supported on some affine hyperplane of $H_{1}(M, \R)$. As $\E(\mu_{o,ab})=\lambda_{\mu_{o}}e_{\nu_{o}}=0$, this hyperplane contains $o$. We now get a contradiction, since the support of $\mu_{o}$ generates $\Gamma$ as a group and $\pi(\Gamma)$ spans $H_{1}(M, \R)$. 

\end{proof}

\details{The discussion at the begining of the section, combined with the two previous lemmas, justifies \Cref{laws-hyperbolic} in the case where $m=\nu_{o}$. We now extend the family of measures that are allowed to obtain the full statement of \Cref{laws-hyperbolic}. We let  $m$ be a probability measure on $\partial \Hn$ which is \emph{absolutely continuous} with respect to the class of Patterson-Sullivan. The LLN and LIL stated in \Cref{laws-hyperbolic} for $m$-typical geodesic rays follow immediately from the case $m=\nu_{o}$, as they only refer to almost-sure behavior. }

Let us check the central limit theorem and the gambler's ruin for geodesic rays chosen via $m$. We use the following decomposition lemma. 

\begin{lemma} \label{reconstruction}
Let $\varphi : \partial \Hn\rightarrow \R_{>0}$ be a smooth positive function (for $\partial \Hn\simeq \mathbb{S}^{N-1}$). Then there exist  coefficients $(c_{\gamma})_{\gamma\in \Gamma}$  such that $c_{\gamma}\in \R^+$ and for all $\xi\in \Lambda_{\Gamma}$,  $$\sum_{\gamma\in \Gamma}c_{\gamma}e^{-\delta_{\Gamma} B_{\xi}(\gamma o, o)}=\varphi(\xi)$$ 
\end{lemma}

\begin{proof}This follows from the appendix in \cite{LiNaudPan21}: just replace $R_{0}=1$ by $R_{0}=\varphi$ in \cite[Section A.3]{LiNaudPan21}.
\end{proof}

\begin{lemma}
$m$-typical geodesic rays satisfy a non-degenerate CLT, with same covariance matrix as $\nu_{o}$, as well the Gamler's ruin estimate.
\end{lemma}

\begin{proof}
 We have proven the CLT and the GR for $m \in \{\nu_{\gamma. o}, \gamma\in \Gamma\}$. Observe that they remain true (with same covariance matrix) for any convex combination of these measures, i.e. for measures of the form 
$$\sum_{\gamma\in \Gamma} c_{\gamma} \nu_{\gamma.o}=\left(\sum_{\gamma\in \Gamma}c_{\gamma}e^{-\delta_{\Gamma} B_{\xi}(\gamma o, o)}\right) \nu_{o}$$ 
where $c_{\gamma}\geq 0$, $\sum c_{\gamma}=1$. By \Cref{reconstruction}, the sum in the right hand side can be chosen to be any $\varphi$   positive smooth function on the boundary $\partial \Hn$. Now, if $m=\psi \nu_{o}$ with $\psi\in L^1(\nu_{o})$ non negative, we can approximate $\psi$ in    $L^1(\nu_{o})$ by such functions $\varphi$ to  extend the result to $m$. 

\end{proof}

It remains to check the PLD stated in \Cref{laws-hyperbolic} in case $m$ has $L^{1+\eps}$ Radon-Nikodym derivatives with respect to Patterson-Sullivan measures.

\begin{lemma}
If $m=\psi \nu_{o}$ where $\psi \in L^{1+\eps}(\partial \Hn, \nu_{o})$ for some $\eps>0$, then $m$-typical geodesic rays satisify the principle of large deviations. 
\end{lemma}

\begin{proof}
 Given $\alpha>0$, set $$E_{t}= \{\xi : \|i\circ r_{\xi}(t)\|\geq \alpha t \}  $$
By  H\"older's inequality, 
$$m(E_{t})=\int_{\Lambda_{\Gamma}} {\bf 1}_{E_{t}}(\xi) \psi(\xi) d\nu_{o}(\xi)\leq \nu_{o}(E_{t})^{\eps/(1+\eps)} \|\psi\|_{L^{1+\eps}(\nu_{o})} $$
whence exponential decay by  the PLD for $\nu_{o}$ (consequence of \Cref{laws-projected} and \Cref{zero-drift}). 
\end{proof}
\bigskip

\subsection{Singularity of harmonic measures} \label{Sec-sing}

To conclude the section,  we combine \Cref{laws-projected} and a result of Enriquez-Franchi-Le Jan \cite{EnriquezFranchiLeJan01} to prove \Cref{singularity}:   Furstenberg measures for $L^1$-random walks on a finitely generated discrete subgroup $\Gamma$ of $PSL_{2}(\R)$ containing a parabolic element  are singular with Patterson-Sullivan measures.

\bigskip
Observe that we authorize $\Gamma\backslash \mathbb{H}^2$ to have only one cusp, of trivial homology. In order to apply \Cref{laws-projected}, we need to generate cusps with non-trivial homology. This is the role of

\begin{lemma} \label{singularity-reduction}
To prove \Cref{singularity}, we can assume that $\Gamma$ is torsion free and $\Gamma \smallsetminus [\Gamma, \Gamma]$ contains a parabolic element. 
\end{lemma}

\begin{proof} 
The idea is to replace  $\mu$ by its induction on a finite index subgroup of $\Gamma$.

By Selberg's Theorem, $\Gamma$ admits a normal subgroup $\Gamma'$ of finite index which is torsion-free. In particular,  $M'=\Gamma'\backslash \mathbb{H}^2$ is  a hyperbolic surface of finite type with a cusp. Its topology is characterized by its genus $g$ and its number of punctures $p\geq 1$. The first homology group is known to have dimension $2g+p-1$. More precisely, a basis consists in the  $2g$ standard loops of the genus $g$ compactification of $M'$, and any $p-1$ loops which are homotopic to different cusps. Up to replacing $M'$ by a $2$-cover, one may assume  $p\geq 2$. In this case, cusps have non-trivial homology in $M'$. This means that $\Gamma' \smallsetminus [\Gamma', \Gamma']$ contains a parabolic element. 

Let $\tau : B\rightarrow \N\cup\{\infty\}, b \mapsto \inf\{k\geq 1, b_{1}\dots b_{k}\in \Gamma'\}$
be the hitting time of $\Gamma'$ for the right $\mu$-walk on $\Gamma$. As $\Gamma'\backslash \Gamma$ is finite, $\tau$ must have finite first moment. Introduce the measure $\mu_{\tau}$, whose distribution is given by $b_{1}\dots b_{\tau(b)}$. By Wald's formula, $\mu_{\tau}$ has  finite first moment  in any word metric on $\Gamma$, hence on $\Gamma'$.  It is also non-elementary, for instance by \cite[Theorem A]{Choi21-CLT}.  Finally, we know from  \cite[Appendix A.2]{Ben22-TdR} that $\mu_{\tau}$ and $\mu$ share the same Furstenberg measure on the boundary $\partial \mathbb{H}^2$. Summing up, we can replace $\mu$ by $\mu_{\tau}$ to prove   \Cref{singularity}, whence the announced reduction.
\end{proof}

We now conclude the proof of \Cref{singularity}. 

\begin{proof}[Proof of \Cref{singularity}]
As in \cite{GuiLeJ93}, followed by \cite{GadreMaherTiozzo15, RandeckerTiozzo21}, the idea is to show that the Furstenberg measure of $\mu$ and the Patterson-Sullivan measure drive geodesic rays with different winding statistics on $M=\Gamma\backslash \mathbb{H}^2$.

In view of \Cref{singularity-reduction}, we may assume that $\Gamma$ is torsion-free and  let $u\in \Gamma\smallsetminus [\Gamma, \Gamma]$ be a parabolic element.
$u$ corresponds  to a loop $l_{u}$ around a cusp of $M$ which is non trivial in $H_{1}(M, \R)$.  By the De Rham duality theorem,  there exists   a closed $1$-form $\omega$ on $M$ with non zero integral on $l_{u}$. We write $\widetilde{\omega}$ its lift to $\mathbb{H}^2$ and define the $\Gamma$-equivariant map $i:\mathbb{H}^2 \to \R$, 
$$ i(z)=\int_{[o,z]}\widetilde{\omega}$$ 
recording the homological winding around the cusp of $u$. Observe that $i$ is locally bounded (though not quasi-Lipschitz). 

Write $\nu$ the Furstenberg measure of $\mu$ on $\partial \mathbb{H}^2$. By the weak law of large numbers  in \Cref{laws-projected-weak}, we have for any $\alpha>0$, 
$$ \nu(\xi\,:\, | i \circ r_{\xi}(t)-te_{\nu}| \leq \alpha t) \rightarrow 1$$

Let $\delta_{\Gamma}$ the exponent of $\Gamma$ (defined at the beginning of \Cref{Sec-winding-hyperbolic}). Note that $\delta_{\Gamma}\in (1/2, 1]$ since $M$ has a cusp \cite[p. 265]{Sullivan84}. Morever $\delta_{\Gamma}=1$ if and only if $M$ has no funnel. Let $D\subseteq \mathbb{H}^2$ be a fundamental domain for $\Gamma$, and $\nu_{D}$ the geodesic projection\footnote{This means the limit of the push-forwards by the geodesic flow $\phi_{t}$ as $t\to+\infty$} of $\widetilde{m}_{BMS}$ on the boundary $\partial \mathbb{H}^2$.  By \cite{EnriquezFranchiLeJan01}, the winding of the geodesic flow directed by $\nu_{D}$ satisfies a central limit theorem with renormalization in $1/t^{\frac{1}{2\delta_{\Gamma}-1}}$: for any $c\in \R$, 
\begin{align*}
 \nu_{D}(\xi \,:\, i \circ r_{\xi} (t) \geq ct^{\frac{1}{2\delta_{\Gamma}-1}})    \rightarrow \mathfrak{S}([c, +\infty)) 
 \end{align*}
where is $\mathfrak{S}$ some atom-free probability distribution  on $\R$ (more precisely a two-sided stable law of exponent $2\delta_{\Gamma}-1\in (0,1)$) and independent of $D$. Up to averaging on several domains $D$, this is still true for some probability $\nu_{1}$ in the class of Patterson-Sullivan. In particular,  for any $\eps>0$, there is $\alpha>0$ such that 
$$ \limsup_{t\to+\infty} \nu_{1}(\xi \,:\,  |i \circ r_{\xi}(t)-te_{\nu}| \leq \alpha t) \leq \eps$$
 This proves that $\nu$ and $\nu_{1}$  are mutually singular.

\end{proof}

\section{Open questions} \label{Sec-questions}

We mention a few problems that were not discussed here, but might lead to interesting continuations of our study.

\begin{itemize}
\item (Nilpotent case) Let $\Gamma$ be a  non-abelian free group, $N$ a simply connected \emph{nilpotent} Lie group, and $\pi : \Gamma\rightarrow N$ a group morphism. Let $\nu$ be a Furstenberg measure on the Gromov boundary $\partial \Gamma$, determined by a finitely supported non-elementary probability measure $\mu$ on $\Gamma$. What are the statistics of $\nu$-typical geodesic rays projected to $N$? Note that limit theorems on $N$ for i.i.d random walks are available \cite{Raugi78, Breuillard05, DiaconisHough21, Hough19, BenardBreuillardCLT, BenardBreuillardLLT}.

\item (Inverse problem) Let $\mu_{1}, \mu_{2}$ be two finitely supported non-elementary measures on a non-abelian free group $\Gamma$.  Is there a simple necessary and sufficient condition characterizing that $\mu_{1}$ and $\mu_{2}$ have the same Furstenberg measure?

\end{itemize}

\appendix

\section{Appendix  :   General limit laws for  cocycles} \label{Sec-general-laws}

We prove general limit theorems for cocycles. Similar results have been studied  for the Iwasawa cocycle on linear groups in \cite{BQRW, BQ16-CLTlin}  and the Busemann cocycle on hyperbolic groups   \cite{BQ16-CLThyp}, but they do not apply directly to the context of the paper. Our motivation is to propose an abstract setting, encapsulating both  \cite{BQRW, BQ16-CLTlin, BQ16-CLThyp} and our needs, and which will hopefully apply to other situations.

\bigskip

We fix   a compact metric space $X$ and  a measurable group $G$ acting measurably on $X$. 

We let $\sigma : G\times X\rightarrow \R^d$ be a measurable cocycle and $|.|:G\rightarrow \R_{\geq0}$ a subadditive \details{measurable} function such that for all $g\in G$, 
$$\sup_{x\in X} \|\sigma(g,x)\|\leq |g|$$
We also fix  a probability measure $\mu$  on $G$ such that $\int_{G}|g|d\mu(g)<\infty$.

\details{
Throughout Sections A.1-A.4, we assume  that $\sigma$ has constant zero-drift for $\mu$: for all $x\in X$, $$\int_{G}\sigma(g,x) d\mu(g)=0$$
Under suitable moment and regularity conditions, we establish a central limit theorem with stopping time, the law of the iterated logarithm, the large deviation principle and a gambler's ruin estimate for the cocycle $\sigma$. 

In  Section A.5, the zero mean condition on $\sigma$ is replaced  by assuming a uniform lower bound on the drift: for all $x\in X$, 
\begin{align*}
\int_{G}\sigma(g,x)d\mu(g) \geq 1  
\end{align*}
and we study the first time for which $\sigma$ leaves a bounded interval under the $\mu$-random walk, as well as the overshoot.
}

\bigskip
\noindent{\bf Remarks}. 
\details{1) If $\sigma(g,.)$ is continuous for every $g\in G$ (as will be assumed in most of this appendix), then one can always choose the function $|.|$ to be $|g|=\sup_{x\in X}\|\sigma(g,x)\|$. It is indeed measurable because the supremum over $X$ can be replaced by the supremum over a countable dense subset of $X$, and subadditivity follows from the cocycle relation.}

\details{
2)  Although formulated for $\sigma$ with constant zero drift, Sections A.1-A.4  yield limit theorems for cocycles $\sigma$ of the form
$$\sigma(g,x) =\sigma_{\lambda}(g,x) +\psi(gx)-\psi(x)$$
where $\sigma$ has constant drift $\lambda$  with $\lambda\in \R^d$ possibly non-zero, and  $\psi :X\rightarrow \R^d$ is measurable and bounded. Indeed, notice that limit theorems for $\sigma_{\lambda}$ follow from the centered case by setting $G'=G\times \Z$ acting on $X$ via the $G$ component,  $\sigma'((g,k),x)=\sigma_{\lambda}(g,x)-k\lambda$, and $\mu'=\mu\otimes \delta_{1}$. It is also obvious that a fixed bounded function $\psi$  does not play any role for Sections A.1-A.3, where the cocycle is renormalized by a term that goes to $+\infty$.  For A.4, one can argue as in the proof of GR in \Cref{laws-projected} to erase the influence of $\psi$. 
}

\bigskip

 We set $(B, \beta)=(G^{\N^*}, \mu^{\otimes \N^*})$, and  write $\mathcal{B}_{k}$ the $\sigma$-algebra of $B$ generated by the first $k$ coordinates. We say that $\mu$ has finite $p$-th moment if $\int_{G}|g|^pd\mu(g)<\infty$. 

\subsection{Central limit theorem} \label{Sec-general-CLT}

We introduce a notion of continuity on $\sigma$ with respect to the space variable in $X$. More precisely, \details{we say $\sigma$ is \emph{$(\mu,X)$-controlled} } if
\begin{itemize}
 \item for every $g\in G$, the function  $\sigma(g, .)$ is continuous.
  \item $\forall \eps>0$, $\exists R>0$, $\forall n\geq 1$ $\forall x,y\in X$, 
  $$
  \details{\mu^{*n}(g\in G\,:\,|\sigma(g,x) -\sigma(g,y)|\leq R)\geq 1-\eps  }
  $$
\end{itemize}

\bigskip

Under this condition, we prove the central limit theorem for $\sigma$. 

\begin{prop} \label{general-CLT}
Assume $\mu$ has finite second moment, and $\sigma$ has constant zero drift and is $(\mu,X)$-controlled. 
Then, 

\begin{enumerate}
\item\emph{(Covariance)}  Uniformly in $x\in X$, we have the convergence 
$$\frac{1}{n} \int_{G}\sigma(g,x)\, {^t \sigma(g,x)} \,d\mu^{*n}(g)\longrightarrow A $$
where  $A\in M_{d}(\R)$ is some positive semi-definite  symmetric matrix.

\item\emph{(CLT)}  Uniformly in $x  \in X$, as $g$ varies with law $\mu^{*n}$,
$$ \frac{\sigma(g,x)}{\sqrt{n}} \underset{n\to +\infty}{\wconv} \mathscr{N}(0, A) $$ 

\end{enumerate}
\end{prop}

\bigskip

\noindent{\bf Remark}. 
1) We see from the proof that the limit covariance matrix $A$ satisfies 
$$A= \frac{1}{n} \int_{G\times X} \sigma(g,x)\,{^t}\sigma(g,x)\,d\mu^{*n} (g)d\nu (x)  $$
for any $n\geq 1$, any $\mu$-stationary probability measure $\nu$ on $X$. One says that $\sigma$ has \emph{unique covariance}.  

2) Benoist and Quint  prove a  CLT where the $(\mu, X)$-control assumption is replaced by the condition that $\sigma$ is continuous with unique covariance \cite[Theorem 3.4]{BQ16-CLTlin}. The $(\mu,X)$-control assumption seems to be more convenient. For instance, it applies directly to both the linear case and the hyperbolic case, wheras establishing unique covariance in the hyperbolic setting requires to go through the whole proof of the CLT, as in \cite[Theorem 4.7]{BQ16-CLThyp}.

\bigskip

For this paper, we actually \emph{need a stronger version of the CLT}, dealing not only with $\mu^{*n}$ but with $\mu^{*\tau_{n}}$ where $\tau_{n}$ is a stopping time distributed around $n$ at scale $\sqrt{n}$. Let us define this notation.

Given $\tau :B\rightarrow \N\cup \{\infty\}$  be a $\beta$-a.e. finite stopping time, we define $\mu^{*\tau}$ as the distribution of the variable $b_{\tau(b)}\dots b_{1}$ when $b$ varies with law $\beta$. 

We say that a family $(\tau_{n})_{n\geq 1}$ of stopping times is \emph{tight around $n$} if for all $\eps>0$, there exists $R>0$ such that 
$$\inf_{n\geq 1} \beta\{\tau_{n} \in [n-R\sqrt{n}, \,n+R\sqrt{n}]\}\geq 1-\eps $$

\begin{prop}[CLT with stopping times] \label{general-CLT+}
Assume $\mu$ has finite second moment, and $\sigma$ has constant zero drift and is $(\mu,X)$-controlled.   Let $(\tau_{n})_{n\geq1}$ be a family  of stopping times which is tight around $n$. Then uniformly in $x  \in X$, as $g$ varies with law $\mu^{*{\tau_{n}}}$,
$$ \frac{\sigma(g,x)}{\sqrt{n}} \underset{n\to +\infty}{\wconv} \mathscr{N}(0, A) $$ 
\end{prop}

\bigskip

\bigskip
We now engage in the proof of the CLT with no stopping time, i.e. \Cref{general-CLT}.  We will combine methods used in  \cite{BQ16-CLThyp} and \cite{BQRW}. Using Cramér-Wold device and the $(\mu,X)$-control assumption (for the uniformity), it is enough to check those statements when $\sigma$ is replaced by its image under a linear form. In other words, \emph{we may assume that $d=1$}. We start by a non-uniform version of the claims. 

\begin{lemma} \label{CLT-gen-lem}
There exists $x_{0}\in X$ and $A\in \R^+$ such that the following holds.  
  \begin{align} \label{var-gen}
\frac{1}{n} \int_{G}  \sigma(g,x_{0})^2 \,d\mu^{*n}(g)\longrightarrow A
  \end{align}
and as $g$ varies with law $\mu^{*n}$,
  \begin{align} \label{CLT-gen}
 \frac{\sigma(g,x_{0})}{\sqrt{n}} \underset{n\to +\infty}{\wconv} \mathscr{N}(0, A) 
   \end{align} 
\end{lemma} 

\begin{proof}
Let $x_{0}\in X$ be a parameter and for $k\geq 1$, $b \in B$, introduce the martingale difference
  \begin{align*}
  X_{k}(b) &= \sigma(b_{k}, b_{k-1}\dots b_{1}x_{0})
  \end{align*}
We only need check to that for some good choice of $x_{0}$, there exists $A\geq0$ such that
  \begin{itemize}
 \item[(i)]  $W_{n}:= \frac{1}{n}\sum_{k =1}^n \E(X_{k}^2|\mathcal{B}_{k-1}) \longrightarrow A$ in $L^1$
 \item[(ii)] For all $\eps>0$,  $W_{\eps, n}:= \frac{1}{n} \sum_{k= 1}^n \E(X_{k}^2 {\bf 1}_{\{|X_{k}|>\eps\}} |\mathcal{B}_{k-1}) \longrightarrow 0$ in probability
  \end{itemize}
  Indeed, this is sufficient: Brown's CLT for martingales \cite{Brown71} yields the convergence $\frac{1}{\sqrt{n}}\sum_{i}X_{i} \to  \mathscr{N}(0, A)$ which is \eqref{CLT-gen}, and integrating item (i) yields  \eqref{var-gen} (use the zero drift condition and the cocycle relation).
  
  Let us check (i). We observe that  $W_{n}(b)= \frac{1}{n}\sum_{1\leq k\leq n} M(b_{k-1}\dots b_{1}x_{0})$ where $M:X\rightarrow \R^+$ is the bounded function defined by
  \begin{align*}
  M(y) =\int_{G}\sigma(g,y)^2d\mu(g)  \tag{$y\in X$}
   \end{align*}
   
   \details{ Choose $\nu$ to be any ergodic  $\mu$-stationary probability measure on $X$. In other terms, the dynamical system $(B\times X, \beta\otimes \nu, T^X)$ where $T^X: (b_{i})_{i\geq 1},x) \mapsto ((b_{i})_{i\geq 2},b_{1}x)$ is measure preserving ergodic \cite[Proposition 2.9]{BQRW}. By Birkhoff Ergodic Theorem, for all $x$ in a set $E$ of $\nu$-measure $1$, $\beta$-almost every $b$, one has $W_{n}(b)\rightarrow \nu(M)$ where $\nu(M)$ is the $\nu$-average of $M$, that is 
      \begin{align*}
\nu(M)=\int_{G\times X} \sigma(g,x)^2\,d\mu(g)d\nu (x)  =:A
  \end{align*}
  To get (i), choose $x_{0}$ in $E$, and note that almost sure convergence implies convergence in $L^1(B,\beta)$ for  sequence of functions that are uniformly bounded. }

For item (ii),  observe the equality $W_{\eps, n}(b)= \frac{1}{n}\sum_{1\leq k\leq n}u_{n}(b_{k-1}\dots b_{1}x_{0})$ where 
  \begin{align*}
  u_{n}(y) &=\int_{G} \sigma(g,y)^2 {\bf 1}_{\{|\sigma(g, y)| > \eps\sqrt{n} \}} \,d\mu(g)  \tag{$y\in X$}\\
  &\leq \int_{G} |g|^2 {\bf 1}_{\{|g| > \eps\sqrt{n} \}} \,d\mu(g)  
   \end{align*}
  By the $L^2$ assumption on $\mu$, we have  $u_{n}(y)\rightarrow 0$ uniformly in $y$. Hence  $W_{\eps, n}\to 0$. 
  \end{proof}

We now upgrade \Cref{CLT-gen-lem} to \Cref{general-CLT}.

\begin{proof}[Proof of \Cref{general-CLT}]
Let $x_{0}\in X$ and $A$ be as in  \Cref{CLT-gen-lem}. The uniform CLT follows directly from \eqref{CLT-gen} and the  $(\mu,X)$-control assumption on $\sigma$. 

To check the convergence of variances, we  must prove that 
  \begin{align*} 
\frac{1}{n} \int_{G}  \sigma(g,x)^2 \,d\mu^{*n}(g)\longrightarrow A
  \end{align*}
 uniformly in $x\in X$. Using the cocycle relation and the zero drift assumption, we can rewrite the left-hand side 
   \begin{align*} 
\frac{1}{n} \int_{G}  \sigma(g,x)^2 \,d\mu^{*n}(g)= \underbrace{\left(\frac{1}{n}\sum_{ k=0}^{n-1}\mu^{*k}*\delta_{x}\right)}_{\kappa_{n,x}}(M)
  \end{align*}
where $M$ is the   function defined previously. To conclude, we need to check the following lemma. 
\details{
\begin{lemma} \label{lemma-nuM}
The sequence of empirical averages $(\kappa_{n,x}(M))_{n\geq 1}$  converges to $A$ uniformly in $x\in X$. 
\end{lemma}
\begin{proof}[Proof of \Cref{lemma-nuM}]
We first observe that for any $\mu$-stationary probability measure $\nu'$ on 
$X$, the $\nu'$-average of $M$ is equal to $A$. Indeed, one can assume $\nu'$ ergodic. The proof of \Cref{CLT-gen-lem} shows that $\nu'(M)$ is the limit covariance in the CLT for $\sigma(., x'_{0})$ where $x'_{0}$ is typical for $\nu'$. But we also note that the $(\mu,X)$-control assumption on $\sigma$ implies that a CLT for $\sigma(., x)$ implies the same CLT  for $\sigma(., y)$ where $y\in X$ is any other spatial coordinate. Hence, the limit covariance is independent of $\nu'$. This justifies that $\nu'(M)=A$ for all stationary probability measure  $\nu'$ on $X$.

To conclude, we now assume by contradiction that $\kappa_{n,x}(M)$ does not converge to $A$ uniformly in $x\in X$, i.e. we have some sequence $\kappa_{n,x_{n}}(M)$ at distance $\eps>0$ from $A$. But $M$ is bounded continous and any accumulation point of  $\kappa_{n,x_{n}}$ is stationary, whence a contradiction with the first paragraph. 
\end{proof}
This concludes the proof of  \Cref{general-CLT}.}
\end{proof}

\bigskip
We now prove the  CLT with  stopping times.

\begin{proof}[Proof of \Cref{general-CLT+}]
 
 Let $R>0$ be some (large) fixed parameter and set, $q_{n}=\lfloor n+R\sqrt{n} \rfloor$, then $\tau'_{n}=\min(\tau_{n}, q_{n})$. By the tightness assumption,  $\tau'_{n}=\tau_{n}$ on a set of $\beta$-measure arbitrarily close to $1$ as long as $R$ is large enough.  Hence, it is sufficient to check the above convergence for the stopping times $\tau'_{n}$, or in other words, we may assume $\tau_{n}\leq q_{n}$.  Now we can use the cocycle property to write
\begin{align*}
\frac{\sigma(b_{\tau_{n}(b)}\dots b_{1}, x)}{\sqrt{n}}  =\frac{\sigma (b_{q_{n}}\dots b_{1},x)}{\sqrt{n}}- \frac{\sigma (b_{q_{n}}\dots b_{\tau_{n}(b)+1},\,b_{\tau_{n}(b)}\dots b_{1}x)}{\sqrt{n}}
\end{align*}
By the CLT in \Cref{general-CLT}, the first term on the right-hand side  converges in law to $\mathscr{N}(0, A)$, uniformly in $x\in X$.  We only need to check that the second term goes to $0$ in probability and uniformly in $x$. Denote by $T:B\rightarrow B, (b_{i})_{i\geq 1}, \mapsto (b_{i+1})_{i\geq 1}$ the one-sided shift. Observing that $T^{\tau_{n}(b)}b$ has law $\beta$, we just need to check   that as $a, b$ vary independently in $B$ with law $\beta$, we have the convergence in probability
$$\frac{\sigma (a_{q_{n}-\tau_{n}(b)}\dots a_{1},\,y)}{\sqrt{n}} \,\pconv\, 0$$
uniformly in $y\in X$. This follows from the tightness condition on the $\tau_{n}$ and the uniformlity of the CLT in \Cref{general-CLT}.

\end{proof}

\subsection{Law of the iterated logarithm}

We prove the law of the iterated logarithm for cocycles.

\begin{prop} \label{general-LIL}
Assume $\mu$ has finite second moment, $\sigma$ has constant zero drift and is $(\mu,X)$-controlled. Let $A$ be the limit covariance matrix in the CLT (\ref{general-CLT}). 
Then,  for every $x\in X$, for $\beta$-almost every $b\in B$, 
$$\left\{\text{ Accumulation points of } \left(\frac{\sigma(b_{n}\dots b_{1}, x)}{\sqrt{2n \log \log n}}\right)_{n\geq 3}  \right\} = A^{1/2}(\overline{B}(0,1)) $$
where $\overline{B}(0,1)=\{t \in \R^d, \sum_{i}t_{i}^2\leq 1\}$.  
\end{prop}

A similar LIL is presented in the book \cite{BQRW} under the assumption $\mu$ has finite exponential moment. To reach the optimal generality of $\mu$ having only finite second moment, we  build on truncation techniques  used by  de Acosta \cite{deAcosta83} in his  proof of the LIL for sum of i.i.d random variables.

\subsection*{The upper bound}

\details{We now prove the first half of \Cref{general-LIL}: the set of accumulation points of the renormalized cocycle is included in $A^{1/2}(\overline{B}(0,1))$. For this purpose, we may reduce the problem to dimension $1$.}

\begin{lemma}\label{GLIL-up-0}
We may assume without loss of generality that $d=1$ and $A=1$. 
\end{lemma}

\begin{proof}
Fix $x\in X$. If the case  $(d=1, A=1)$ holds, then for every linear form $\varphi$ on $\R^d$, for $\beta$-almost every  $b\in B$, the accumulation points of $\left(a_{n}^{-1}\varphi\circ \sigma(b_{n}\dots b_{1}, x)\right)_{n\geq 3} $ are included in $\varphi A^{1/2}(\overline{B}(0,1))$. Now  use that the inclusion in $A^{1/2}(\overline{B}(0,1))$ can be characterized via a countable set of linear forms on $\R^d$ to conclude. 
\end{proof}

\bigskip

From now on, we write $a_{n}=\sqrt{2n\log \log n}$.

\begin{lemma}\label{GLIL-up-1}
Fix $r>1$ arbitrary. It is enough to show that for all $x\in X$, for     $\beta$-almost every $b\in B$, 
\begin{align}  \label{GLIL-up-1.0}
\limsup \frac{1}{a_{n}} \sigma(b_{n}\dots b_{1}, x) \leq r^3
\end{align} 
\end{lemma}

\begin{proof}
The upper bound in the LIL follows by applying \eqref{GLIL-up-1.0} to $\sigma$ and $-\sigma$, and by letting $r\to1^+$ along a countable sequence. 
\end{proof}

\bigskip
\begin{lemma}\label{GLIL-up-1}
Let $\delta>0$. To prove \eqref{GLIL-up-1.0}, we may assume 
\begin{align}  \label{GLIL-varpar}
\int_{G}\sigma(g,x)^2 \,d\mu(g) \leq 1+\delta 
\end{align} 
\end{lemma}

\begin{proof}
Notice that for fixed $k\geq 1$, and $\beta$-almost every $b\in B$, the sequences 
\begin{align*} 
a_{n}^{-1} \sigma(b_{n}\dots b_{1}, x)\,\,\,\, \,\,\,\, \,\,\,\, a_{nk}^{-1} \sigma(b_{nk}\dots b_{1}, x) \tag{$n\geq 1$}
\end{align*} 
have the same accumulation points. Hence, we can replace $(\sigma, \mu)$ by $(\frac{1}{\sqrt{k}}\sigma, \mu^{*k})$ without loss of generality.  By the assumption $A=1$, and the uniform convergence of variances in \Cref{general-CLT}, there exists $k \geq 1$ such that  for all $x\in X$, 
$$\int_{G}\frac{1}{k}\sigma(g,x)^2 \,d\mu^{*k}(g) \leq 1+\delta $$ 
whence the result.
\end{proof}

\bigskip
We now introduce truncations of $\sigma$. These are crucial to deal with the weak moment condition on $\sigma$. 
 
 Let $\alpha>0$ be some parameter to be specified later. For $j\geq 1$, set $\LL j=\log \log j$. We define  $X^{(j)}_{1}$ as the $\alpha \sqrt{j/\LL j}$-truncation of $\sigma$, and call $X^{(j)}$ the corresponding recentered variable: 
\begin{align*}
X^{(j)}_{1}(g,x) = \sigma(g,x) {\bf 1}_{|g|\leq \alpha \sqrt{\frac{ j}{\LL j}}} \,\,\,\, \,\,\,\, \,\,\,\, \,\,\,\, \,\,\,\,X^{(j)}(g,x) = X^{(j)}_{1}(g,x) - \E_{x}(X^{(j)}_{1})
\end{align*}
where $\E_{x}$ refers to the $\mu\otimes \delta_{x}$-average.   

For $(b,x)\in B\times X$, we set 
$$S_{n}(b,x)=\sum_{j=1}^n X^{(j)}(b_{j},b_{j-1}\dots b_{1}.x) $$
$S_{n}(b,x)$ corresponds to a version of $\sigma(b_{n}\dots b_{1}, x)$ where the $n$-increments arising from the cocycle relation have been truncated with increasing marge of error and (slightly) recentered. We first note that $S_{n}(b,x)$ is a good  approximation of $\sigma(b_{n}\dots b_{1}, x)$, at least for the purpose of proving the LIL:

\begin{lemma}\label{GLIL-up-1}
For every $\alpha>0$, for all $x\in X$, $\beta$-almost every $b\in B$, 
\begin{align}  \label{GLIL-up-app}
\sigma(b_{n}\dots b_{1}, x)= S_{n}(b, x)  +o(a_{n})
\end{align} 
\end{lemma}

\begin{proof} 
Using the inequality $|\sigma(g,x)|\leq |g|$ and the zero drift assumption, we can bound  the difference  
\begin{align}  \label{GLIL-up-approx}
|\sigma(b_{n}\dots b_{1},x) -S_{n}(b,x)|\leq \sum_{j=1}^n |b_{j}| {\bf 1}_{|b_{j}|> \alpha \sqrt{\frac{j}{\LL j}}} +  \sum_{j=1}^n\E( |g| {\bf 1}_{|g|> \alpha \sqrt{\frac{j}{\LL j}}})
\end{align} 
The second moment assumption and \cite[Lemma 2.3]{deAcosta83} imply that   the right-hand side is $\beta$-almost surely of the form $o(a_{n})$. 
\end{proof}

\bigskip

 We now play on   the variance parameter $\delta$ and the truncation parameter $\alpha$ to prove the upper bound \eqref{GLIL-up-1.0}. 

\begin{lemma}\label{GLIL-up-4}
We can choose the parameters $\delta, \alpha>0$ so that for some $c>1$, all $n\geq 1$,
$$\sup_{x\in X} \mathbb{P}_{x}(S_{n}\geq r a_{n}) = O(\exp(-c \LL n))$$
\end{lemma}

\begin{proof}
Using the relation $e^s\leq 1+s+\frac{s^2}{2}e^{|s|}$ for any $s\in \R$ and the fact that $X^{(j)}$ has zero average,   we get for $t>0$, $x\in X$,  
$$\E_{x}(e^{t X^{(j)}}) \leq 1 + \frac{t^2}{2} \E_{x}(|{X^{(j)}}|^2) \exp(t\alpha \frac{j}{\sqrt{\LL j}})$$
The relation \eqref{GLIL-varpar} allows to bound the second moment of $X^{(j)}$: for $j$ large enough (say $j\geq j_{0}$), 
$$\E_{x}(|{X^{(j)}}|^2)\leq \int_{G}\sigma(g,x)^2 \,d\mu(g) + \delta \leq 1+2\delta := a^2$$
hence
$$\E_{x}(e^{t X^{(j)}}) \leq \exp\left(\frac{t^2a^2}{2} \exp(t\alpha  \frac{j}{\sqrt{\LL j}})\right) $$
We can now bound  the exponential moments of $S_{n}$. For $n\geq 1$, $t>0$, using that 
$$\E_{x}(e^{t X^{(j)}(b_{j},b_{j-1}\dots b_{1}x)}\,|\,\mathcal{B}_{j-1} )\leq  \sup_{y\in X} \E_{y}(e^{t X^{(j)}}),$$ we get
\begin{align*}  
\E_{x}(e^{t S_{n}}) &\leq \prod_{j=1}^{n} \sup_{y\in X}\E_{y}(e^{t X^{(j)}})\\
&\leq \exp\left(\frac{t^2a^2}{2} \sum_{j=1}^n\exp(t\alpha  \frac{j}{\sqrt{\LL j}})\right) 
\end{align*}  
So by Chebychev inequality
\begin{align*}  
\mathbb{P}_{x}(S_{n}>r a_{n}) \leq \exp(-tra_{n} +\frac{t^2a^2}{2} \sum_{j=1}^n\exp(t\alpha  \frac{j}{\sqrt{\LL j}}) )
\end{align*} 
Setting $t= 2ra^{-2}\sqrt{2(\LL n)/n}$, a direct computation yields 
\begin{align*}  
\mathbb{P}_{x}(S_{n}>r a_{n}) \leq \exp(- (r/a)^2 (2 - e^{r a^{-2}\alpha})\LL n)
\end{align*} 
As $\delta, \alpha \to 0^+$, we have $c:=(r/a)^2 (2 - e^{r a^{-2}\alpha})\rightarrow r^2 >1$. This concludes the proof. 
\end{proof}

\bigskip

The rest of the proof is  dedicated to showing that for $\delta, \alpha$ as in  \Cref{GLIL-up-4}, we have the upper bound \eqref{GLIL-up-1.0}.

\bigskip
\begin{lemma}\label{GLIL-up-2}
For all $\eps>0$, there exists $n_{0}\geq 0$ such that for $n\geq n_{0}$, $x\in X$, 
$$\min_{0\leq k \leq n}\mathbb{P}_{x}(|S_{n} -S_{k}|\leq \eps a_{n}\,|\,\mathcal{B}_{k})\geq 1/2 $$
\end{lemma}

\begin{proof}
Call $r_{n}(b)$ the right-hand side of \eqref{GLIL-up-approx}. We have 
$$|S_{n} -S_{k}| \leq |\sigma(b_{n}\dots b_{k+1}, b_{k}\dots b_{1}x)|+2r_{n}(b)$$
As $r_{n}(b)=o(a_{n})$ almost-surely,  we only need to check that for large $n$, for all $y\in X$, 
$$\min_{0\leq l \leq n}\beta (|\sigma(b_{l}\dots b_{1}, y) |\leq \frac{\eps}{2} a_{n} )\geq 2/3 $$
For $l$ greater than some $n_{0}'$, this follows from the CLT assumption. For $l\leq n_{0}'$, this follows from the $L^2$-domination assumption.  
\end{proof}

\bigskip
Set $S_{n}^*=\max_{k\leq n}S_{k}$.

\begin{lemma}\label{GLIL-up-3}
For all $c,\eps>0$,   for $n\geq n_{0}$, $x\in X$, 
$$\mathbb{P}_{x}(S_{n}^* \geq (c+\eps)a_{n})\leq 2 \mathbb{P}_{x}(|S_{n}| \geq ca_{n})$$
\end{lemma}

\begin{proof}
It is enough to check that 
$$\mathbb{P}_{x}( S_{n} \geq ca_{n}  \,\,|\,\,S_{n}^* \geq (c+\eps)a_{n})\geq \frac{1}{2}$$
This follows from \Cref{GLIL-up-2}.
\end{proof}

\bigskip

\begin{lemma}\label{GLIL-up-5}
Choose $\delta, \alpha>0$ as in \Cref{GLIL-up-4}. For all $x\in X$, for $\beta$-almost every $b\in B$, 
\begin{align*} 
\limsup \frac{1}{a_{n}}S_{n}(b,x) \leq r^3
\end{align*} 
\end{lemma}

\begin{proof}

 Fix some  $s\in (1, r)$ and set $k_{n} =\lfloor s^n\rfloor$. Writing \emph{i.o.} for \emph{infinitely often}, we have  
\begin{align*} 
\mathbb{P}_{x}(S_{n} \geq r^3 a_{n} \,\, i.o.\,)&\leq \mathbb{P}_{x}(S^*_{k_{n}} \geq r^3 a_{k_{n-1}}  \,\,i.o. \,)\\
&\leq \mathbb{P}_{x}(S^*_{k_{n}}  \geq r^2 a_{k_{n}}  \,\,i.o. \,)
\end{align*} 
where the second inequality uses that $r k_{n-1}\geq k_{n}$ for $n$ large. Let us check that the right-hand side is $0$. Using \Cref{GLIL-up-3} and Borel Cantelli Lemma, we only need to check that 
$$\sum \mathbb{P}_{x}(S_{k_{n}}  \geq r a_{k_{n}})<\infty$$
By \Cref{GLIL-up-3}, 
$$\mathbb{P}_{x}(S_{k_{n}}  \geq r a_{k_{n}})\leq \exp( -c \,\LL k_{n}) \ll n^{-c} $$
whence the finitess. This shows that  
$$\mathbb{P}_{x}(S_{n} \geq r^3 a_{n} \,\, i.o.\,)=0$$
\end{proof}

\noindent{\bf Conclusion of the proof.} We only need to check that \Cref{GLIL-up-1.0} is satisfied. This follows from \Cref{GLIL-up-5} and \Cref{GLIL-up-1}.

\subsection*{The lower bound}

The proof of the reverse inclusion in the LIL follows verbatim the lines of the book \cite[Section 12.4]{BQRW} given in the context of $\mu$-contracting cocycles with finite exponential moments. We choose not to repeat it.

\subsection{Principle of large deviations}

We show the principle of large deviations for  cocycles. The proof is  classic but also very short, so we include it for completeness.

\begin{prop} \label{general-PLD}
Assume $\mu$ has finite exponential moment: $\int_{G}e^{\alpha |g|}\,d\mu(g)<\infty$ for some $\alpha>0$, and $\sigma$ has constant zero drift.  Then for all $\eps>0$ there exists $D, \delta>0$ such that  for any $n\geq 1$,  $x\in X$,  
$$ \beta(b\,:\, \|\sigma(b_{n}\dots b_{1}, x) \|> \eps n ) \leq  D e^{-\delta n}$$
\end{prop}

\begin{proof}
We can assume that $\sigma$ is $\R$-valued. Let $t\in (0, \alpha)$ be a parameter, $n\geq 1$, $x\in X$. Using the cocycle relation and Fubini Theorem, 
 \begin{align*}
  \beta(b\,:\, \sigma(b_{n}\dots b_{1}, x) > \eps n ) &\leq e^{- t \eps n} \int_{B}e^{t \sigma(b_{n}\dots b_{1}, x)} d\beta(b)\\
  &\leq e^{- t \eps n} \left(\sup_{y\in X}\int_{G}e^{t \sigma(g, y)} d\mu(g)\right)^n
 \end{align*}
Using the bound $e^s\leq 1+s+s^{2}e^{|s|}$ for any $s\in \R$, we get  for $y\in X$,
 $$\int_{G}e^{t \sigma(g, y)} d\mu(g) \leq 1 + t^2 C \leq \exp(t^2 C )$$
 where $C= \E_{\mu}(|g|^2e^{t |g| })$ is finite. Combining inequalities, we get
 \begin{align*}
  \beta(b\,:\, \sigma(b_{n}\dots b_{1}, x) > \eps n ) \leq \exp(  - ( \eps- t C) t n)
 \end{align*}
Choosing $t<\eps/C$, we have an exponential bound independent of $x$. Applying the argument to $-\sigma$, we ge the desired bound for $|\sigma(b_{n}\dots b_{1}, x)|$.  
\end{proof}

\subsection{Gambler's ruin}

We show grambler's ruin estimates for  cocycles. 

\begin{prop} \label{general-GR}
Assume $\mu$ has finite exponential moment,  $\sigma$ has constant zero drift and is $(\mu,X)$-controlled,  $d=1$, and the limit variance $A$ in the CLT (\ref{general-CLT}) is non-zero.  Fix $k,l>0$ and sequences $(k_{s})_{s\geq 0},(l_{s})_{s\geq 0} \in \R^\N$ such that $k_{s}\sim sk$, $l_{s}\sim sl$. Then for every $x\in X$, 
$$\beta \left\{b : (\sigma(b_{n}\dots b_{1}, x))_{n\geq 1} \text{ meets  $[l_{s}, +\infty)$ before $(-\infty, -k_{s}]$}   \right\} \,\underset{s\to+\infty}{\longrightarrow}\, \frac{k}{k+l}$$
\end{prop}

\bigskip

For the proof, we  fix $x\in X$ and write $M_{n}(b)=\sigma(b_{n}\dots b_{1}, x)$ the associated martingale.

We will need to know that $M_{n}$ does leave the window $(-k_{s}, l_{s})$, and the exit time is bounded polynomially in $s$ with high probability.

\begin{lemma} \label{exit} 
$$\beta\{b: (M_{n}(b))_{n\geq1} \text{ leaves $(-k_{s},l_{s})$ for some $n\leq s^3$}\}\underset{s\to +\infty}{\longrightarrow} 1$$
\end{lemma} 

\begin{proof}
\begin{align*}
\beta\{b: M_{\lfloor s^3\rfloor}(b) \in (-k_{s},l_{s})\}
&= \beta\{b: s^{-3/2}M_{\lfloor s^3\rfloor}(b) \in (-s^{-3/2}k_{s}, s^{-3/2}l_{s})\} \\
&\underset{s\to +\infty}{\longrightarrow} 0
\end{align*}
by the CLT for $M_{n}$ proven in \Cref{general-CLT}.  
\end{proof}

\bigskip

\begin{proof}[Proof of \Cref{general-GR}]
Set 
$$\tau'_{s}:=\inf\{n\geq 1, \,(M_{n}(b))_{n\geq 1} \text{ meets  $[l_{s}, +\infty)$ or $(-\infty, -k_{s}]$}\}$$
and $\tau_{s}:=\min(\lfloor s^3 \rfloor, \tau'_{s})$.  By the optional sampling theorem applied to the zero-mean martingale $M_{n}$,  we have 
\begin{align*}
0 &=\E(M_{\tau_{s}})\\
&= \E(M_{\tau_{s}}{ \bf1}_{M_{\tau_{s}} \geq l_{s}})+\E(M_{\tau_{s}} { \bf1}_{M_{\tau_{s}} \leq -k_{s}})+\E(M_{\tau_{s}}{ \bf1}_{M_{\tau_{s}}\in (-k_{s}, l_{s})})
\end{align*}

Let us estimate each term individually.
\begin{align*} 
\E(M_{\tau_{s}}{ \bf1}_{M_{\tau_{s}} \geq l_{s}}) = l_{s}\beta(M_{\tau_{s}} \geq l_{s})+ \E((M_{\tau_{s}}-l_{s}){ \bf1}_{M_{\tau_{s}} \geq l_{s}}) 
\end{align*}
 \details{We show the overshoot is negligible compared to $s$. We first notice the upper bound
\begin{align} 
 \E((M_{\tau_{s}}-l_{s}){ \bf1}_{M_{\tau_{s}}  \nonumber
 \geq l_{s}})& \leq \E(\sup_{k\leq s^3}|M_{k+1}-M_{k}|)\\ \nonumber
 & =\int_{\R^+}  \beta(\sup_{k\leq s^3}|M_{k+1}-M_{k}| > t) \,dt\\ \nonumber
 & =\int_{\R^+} 1- \beta(\sup_{k\leq s^3}|M_{k+1}-M_{k}| \leq t) \,dt \\ 
 & =\int_{\R^+} 1- \beta(E_{s^3}(t)) \,dt  \label{ineq-ovs}
\end{align}
where for $r>0$, we write $E_{r}(t):=\{b\in B\,:\, \sup_{k\leq r}|M_{k+1}(b)-M_{k}(b)| \leq t\}$.  Arguing conditionally and observing that the cocycle relation gives  $M_{k+1}(b)-M_{k}(b)=\sigma(b_{k+1}, b_{k}\dots b_{1}x)$, we get 
\begin{align*} 
\beta(E_{s^3}(t)) 
= \beta(|M_{\lfloor s^3\rfloor+1}-M_{\lfloor s^3\rfloor}| \leq t \,|\, E_{s^3-1}(t)) \beta(E_{s^3-1}(t))
 \geq (1-e^{-\alpha t})  \beta(E_{s^3-1}(t))
\end{align*}
for some $\alpha=\alpha(\mu, |.|)$ coming from the exponential moment assumption on $\mu$. Iterating this inequality and plugging it into \eqref{ineq-ovs}, we get
\begin{align*} 
\E((M_{\tau_{s}}-l_{s}){ \bf1}_{M_{\tau_{s}} \geq l_{s}})
   \leq \int_{\R^+} 1- (1-e^{-\alpha t})^{s^3} \,dt
\end{align*}
 \Cref{integral-estimate} below tells us that $\sup_{s>0}\int_{[\eps s, +\infty]} 1- (1-e^{-\alpha t})^{s^3} \,dt <\infty$ for any $\eps>0$, whence we can finally control the overshoot by:
\begin{align*} 
\E((M_{\tau_{s}}-l_{s}){ \bf1}_{M_{\tau_{s}} \geq l_{s}})
 =o(s)
\end{align*}}
To sum up, we proved that 
\begin{align*} 
\E(M_{\tau_{s}}{ \bf1}_{M_{\tau_{s}} \geq l_{s}}) = l_{s}\beta(M_{\tau_{s}} \geq l_{s})+ o(s)
\end{align*}
Using either the same line of argument, or \Cref{exit}, we get as well that 
\begin{align*} 
&\E(M_{\tau_{s}}{ \bf1}_{M_{\tau_{s}} \leq -k_{s}}) = -k_{s}\beta(M_{\tau_{s}} \leq -k_{s})+ o(s)\\
& \details{|\E(M_{\tau_{s}}{ \bf1}_{M_{\tau_{s}}\in (-k_{s}, l_{s})})|\leq \max(k_{s}, l_{s}) \beta(M_{\tau_{s}}\in (-k_{s}, l_{s}))=o(s)}
\end{align*}
In conclusion, 
\begin{align*}
0 &=\E(M_{\tau_{s}})\\
&= l_{s}\beta(M_{\tau_{s}} \geq l_{s}) -k_{s}\beta(M_{\tau_{s}} \leq -k_{s}) +o(s)
\end{align*}
so as $s\to+\infty$, one has $\beta(M_{\tau_{s}}\geq l_{s})\rightarrow \frac{k}{k+l}$ as desired. \details{The result for $\tau'_{s}$ follows because $\beta(\tau_{s}\neq \tau'_{s})\to 0$ as $s\to+\infty$ by \Cref{exit}.}
\end{proof}

\details{We record the following estimate used in the proof  of \Cref{general-GR} above. 
\begin{lemma} \label{integral-estimate}
For all $\eps>0$, one has 
$$\sup_{s>0}\int_{[\eps s, +\infty]} 1- (1-e^{-\alpha t})^{s^3} \,dt <\infty.$$
\end{lemma}}

\begin{proof}
The integral over the larger domain $\int_{\R^+} 1- (1-e^{-\alpha t})^{s^3} \,dt$ is finite for every $s>0$ and increases with $s$. Hence, we only need to check the bound for $s$ large enough. Now rewrite uniformly for $t\in [\eps s, +\infty)$,  
$$(1-e^{-\alpha t})^{s^3}= \exp(s^3\log(1-e^{-\alpha t})) =\exp(s^3 (e^{-\alpha t} +o(e^{-\alpha t})  )) = 1+ s^3e^{-\alpha t} +o(s^3e^{-\alpha t} )$$
where $o(.)$ goes to zero as $s\to+\infty$. The result then follows because 
$$\sup_{s>0}\int_{[\eps s, +\infty]} s^3e^{-\alpha t} \,dt <\infty.$$

\end{proof}

\subsection{Exit estimates for drifting cocycles}

For this last section, we replace the zero mean condition on $\sigma$ by assuming a uniform lower bound on the drift: for all $x\in X$, 
\begin{align} \label{drift+}
\int_{G}\sigma(g,x)d\mu(g) \geq 1  
\end{align}

For $b\in B, x\in  X$,  we set
$$S_{n}(b,x)=\sigma(b_{n}\dots b_{1}, x) $$
and for $s>0$, we write
$$\tau_{s} (b,x)= \inf \{n \geq 1,\, S_{n}(b, x)\geq s\}$$

Foster's criterion (see \cite{Foster53} or \cite[Lemma 2.4]{BenardDeSaxce22}) implies that for any $x$, the exit time $\tau_{s}$ is finite, and even has bounded expectation $\E_{x}(\tau_{s})\leq s$. The goal of the section is to show that  the overshoot $S_{\tau_{s}}-s$ is essentially  bounded. Assuming equality in \eqref{drift+}, we also prove that $\tau_{s}$ lands roughly in a window of size $\sqrt{s}$ around $s$. \details{Throughout the section, given a formula $I(n)$  representing an inequality involving (among others) a parameter $n \in \N^*$, we will lighten notation by writing 
$$ \mathbb{P}_{x}(I(\tau_{s}))=\beta(b\in B\,:\, I(\tau_{s}(b,x)))$$}

\begin{prop} \label{general-tau}
\begin{enumerate}

\item  \emph{(Overshoot control)} \details{Assume that $\sigma$ satisfies \eqref{drift+}.}
$$\sup_{x\in X, s>0}\mathbb{P}_{x}(S_{\tau_{s}} >s +R ) \underset{R\to+\infty}{\longrightarrow} 0$$

\item  \emph{(Time control)} Assume $\mu$ has finite second moment,  and $\sigma$ is $(\mu,X)$-controlled and has constant drift equal to $1$. Then we have:  $\forall \eps>0$, $\exists R>0$, 
\details{$$\sup_{x\in X, s>1} \mathbb{P}_{x}(|\tau_{s} -s| > R\sqrt{s} ) \leq \eps$$}

\end{enumerate}
\end{prop}

\noindent{\bf Remark}. The control of the overshoot  implies that, given any large enough interval $I$ in $\R^+$, the $\beta$-probability that the random sequence $(S_{n})_{n\geq 0}$ intersects $I$ is close to $1$. This can be interpreted as  a weak form of renewal estimate for $(S_{n})_{n\geq 0}$. 
\bigskip

Let us begin with the overshoot estimate. 
Our first step is to use the cocycle property of $S_{n}$ to reduce the overshoot control for $S_{\tau_{s}}$ to that of a stronger control for $S_{\tau_{1}}$. 

\begin{lemma} \label{taus-to-tau1}
For $R>0$, we have
\details{$$\sup_{x\in X, s>1}\mathbb{P}_{x}( S_{\tau_{s}} >s +R ) \,\, \leq \,\,\sum_{n\geq 1} \sup_{y\in X}\Pb_{y}( S_{\tau_{1}} >n +R )$$}
\end{lemma}

\begin{proof}
\details{Fix $R>0$. Let $n\geq 1$,  and assume $H_{n}$ is a constant such that 
\begin{align}\label{eq-induction}
\sup_{x\in X, s\in (1, n]}\Pb_{x}( S_{\tau_{s}} >s +R ) \leq H_{n}
\end{align}
It is enough to show that 
\begin{align}\label{eq-induction2}
\sup_{x\in X, s\in (1, n+1]}\Pb_{x}( S_{\tau_{s}}>s +R ) \leq \sup_{y\in X}\Pb_{y}( S_{\tau_{1}} >n +R ) +H_{n}
\end{align}

Let $x\in X$, $s\in (n, n+1]$. In the event that $S_{\tau_{s}} >s +R$, the sequence $(S_{k})_{k\geq 1}$ does not meet the interval $[s, s+R]$, in particular we get
\begin{align*}
\Pb_{x}( S_{\tau_{s}} >s +R ) &= \Pb_{x}( S_{\tau_{1}}>s +R )+\Pb_{x}( S_{\tau_{s}} >s +R ;\, S_{\tau_{1}}< s)\\ 
&= \Pb_{x}( S_{\tau_{1}}>s +R ) \\
&\,\,\,+\Pb_{x}( S_{\tau_{s}} >s +R \,|\, S_{\tau_{1}} <s)\Pb_{x}(S_{\tau_{1}} <s)\\
&\leq \Pb_{x}( S_{\tau_{1}} >n +R ) +\Pb_{x}( S_{\tau_{s}} >s +R \,|\, S_{\tau_{1}}<s)
\end{align*}}

Let us bound above the second term.  The cocycle property allows to write for $0\leq k\leq n$
$$S_{n}(b, x)=S_{n-k}(T^k b, \,b_{k}\dots b_{1}x) +S_{k}(b, x)$$ 
where we recall that $T:B\rightarrow B, (b_{i})_{i\geq 1}\mapsto (b_{i+1})_{i\geq 1}$ is the one-sided shift. Applying this to $k=\tau_{1}(b,x)$ and $n=\tau_{s}(b,x)$,    we have
\details{\begin{align*}
\Pb_{x}( S_{\tau_{s}} >s +R \,\,|\,\, S_{\tau_{1}} <s) 
&=  \mathbb{P}_{x}( S_{\tau_{s}-\tau_{1}}(T^{\tau_{1}}b, \,b_{\tau_{1}}\dots b_{1}x) >s +R -  S_{\tau_{1}} \,\,|\,\, S_{\tau_{1}} <s)\\
&= \mathbb{P}_{x}( S_{\tau_{s -S_{\tau_{1}}}}(T^{\tau_{1}}b, \,b_{\tau_{1}}\dots b_{1}x) >s +R -  S_{\tau_{1}} \,\,|\,\, S_{\tau_{1}} <s)
\end{align*}}
As  $S_{\tau_{1}}\geq 1$, the lower bound \details{$s +R -  S_{\tau_{1}}$} is in the  initial segment $[0, n+R]$ so we may apply the hypothesis \eqref{eq-induction} to get
\begin{align*}
\Pb_{x}( S_{\tau_{s}} >s +R \,\,|\,\, S_{\tau_{1}} <s) \leq H_{n}
\end{align*}
In sum, we have proven  for any $x\in X$, $s\in (n, n+1]$, 
 \begin{align*}
\Pb_{x}( S_{\tau_{s}}>s +R ) \leq \Pb_{x}( S_{\tau_{1}} >n +R ) +H_{n}
\end{align*}
and the bound  \eqref{eq-induction2} follows, whence the lemma. 
\end{proof}

\bigskip
We now need to show that the sum $\sum_{n\geq 1} \sup_{x\in X}\Pb_{x}( S_{\tau_{1}}>n +R )$ appearing in \Cref{taus-to-tau1} is small as long as $R$ is large. We first check it is finite.

\begin{lemma}  \label{stoch-dom1} \details{Assume that $\sigma$ satisfies \eqref{drift+}.} Then
$$\sum_{n\geq 1}\sup_{x\in X}\Pb_{x}(S_{\tau_{1}} > n) <\infty$$
\end{lemma}

\begin{proof} \details{For $x\in X$, $n\geq 1$, using that $|\sigma(g,x)| \leq |g|$ for all $g\in G$ and that $\tau$ is a stopping time,  we have}
\details{\begin{align*}
\Pb_{x}(S_{\tau_{1}} > n) &= \sum_{k} \Pb_{x}(S_{k}>n,\,\,  \tau_{1}=k)\\
&\leq  \sum_{k} \Pb_{x}(S_{k}-S_{k-1}>n-1, \,\,  \tau_{1}=k)\\
&\leq  \sum_{k} \Pb_{x}(|b_{k}|>n-1, \,\,  \tau_{1}=k)\\
&\leq  \sum_{k} \Pb_{x}(|b_{k}|>n-1, \,\,   \tau_{1}\geq k)\\
&= \mu(|g|>n-1) \sum_{k}\Pb_{x}( \tau_{1}\geq k)\\
&= \mu(|g|>n-1) \E_{x}( \tau_{1})
\end{align*}}
 Foster's criterion (see \cite{Foster53} or \cite[Lemma 2.4]{BenardDeSaxce22}) yields the bound $\E_{x}( \tau_{1})\leq 1$. It follows that 
$$\sum_{n\geq 2}\sup_{x\in X}\Pb_{x}(S_{\tau_{1}} > n) \leq \E_{\mu}(|g|)$$
which finishes the proof. 
\end{proof}

We  now conclude the proof of the overshoot estimate in  \Cref{general-tau}.

\begin{proof}[Proof of \Cref{general-tau} (overshoot control)]
The estimate for the supremum involving only $s\in (0, 1]$ follows from \Cref{stoch-dom1}. Indeed, for $R>1$, 
 $$ \sup_{x\in X, s\in (0,1]}\mathbb{P}_{x}(S_{\tau_{s}} >s +R )\leq \sup_{x\in X}\mathbb{P}_{x}(S_{\tau_{1}} >R )  \underset{R\to+\infty}{\longrightarrow} 0$$

Now assume $s> 1$. Using \Cref{taus-to-tau1}, we just need to check that
$$\sum_{n\geq 1} \sup_{x\in X}\Pb_{x}( S_{\tau_{1}} >n +R )  \underset{R\to+\infty}{\longrightarrow} 0$$
 For $R=0$, this sum is finite  by \Cref{stoch-dom1}, and the result follows by dominated convergence (using again that for $n\geq 0$, $\sup_{x\in X}\mathbb{P}_{x}(S_{\tau_{1}} >n+R )$ goes to $0$ as $R$ goes to $+\infty$ by \Cref{stoch-dom1}).
\end{proof}

\bigskip

We now deal with the time estimate in \Cref{general-tau}. 

\begin{proof}[Proof of \Cref{general-tau} (time control)]

 Let $R>0$ be a parameter to be specified below.

 \details{For $s>1$, set $q_{s} = \lfloor s+R\sqrt{s}\rfloor$. We first show that for large $R>0$, we have
  $$\sup_{x\in X, s>1} \mathbb{P}_{x}( \tau_{s}> q_{s} ) \leq \eps/2$$
}Indeed, if $\tau_{s}(b,x) >q_{s}$, then by definition $S_{q_{s}}(b,x)<s $, so $|S_{q_{s}}(b,x) - q_{s}|>R\sqrt{s}$. The central limit theorem for $S_{n}-n$   (\Cref{general-CLT}) guarantees this event only occurs on a set of measure at most $\eps/2$ as long as $R$ is large enough (independently of $x$ and $s>1$), hence the claim.

\details{ For $s>1$, set $q'_{s} = \lfloor s-R\sqrt{s}\rfloor$.  We  now show that for large $R>0$, we have 
 $$\sup_{x\in X, s>1} \mathbb{P}_{x}( \tau_{s}\leq q'_{s} ) \leq \eps/2$$
Using the cocycle property and the Markov property, we see that for every $\alpha\in(0,1)$, there exists $R'>0$ such that for any $x\in X$, $s>1$, 
$$\mathbb{P}_{x}( S_{q'_{s}} \geq s-R'   \,|\, \tau_{s}\leq q'_{s}  )\, \geq \,1- \sup_{y\in X, n\geq 0}\mathbb{P}_{y}(S_{n}\leq -R')  \geq 1-\alpha$$
where the inequality on the very right is justified by the central limit theorem for $S_{n}-n$  (\Cref{general-CLT}) and the choice of a large enough $R'$. }
This leads to 
\begin{align*}
\mathbb{P}_{x}(\tau_{s}\leq q'_{s}  ) &\leq (1-\alpha)^{-1} \mathbb{P}_{x}( S_{q'_{s}} \geq s-R' ) \\
&\leq (1-\alpha)^{-1} \mathbb{P}_{x}( |S_{q'_{s}} -q'_{s}| \geq R\sqrt{s} -R' )
\end{align*}
Choosing $\alpha$ small enough, $R'$ accordingly, and applying the CLT for $S_{n}-n$ from \Cref{general-CLT} bounds the right-hand side by $\eps/2$ for $R,s$ large enough. We can increase $R$ even more to deal with all $s>1$.

\end{proof}

\small

\bibliographystyle{abbrv}

\bibliography{bibliographie}

\bigskip

\noindent Timothée B\'enard

\noindent\textsc{LAGA - Institut Galilée, 99 avenue Jean Baptiste Clément, 93430 Villetaneuse}

\noindent\textit{Email address}: \texttt{benard@math.univ-paris13.fr}

\end{document}